\documentclass{amsart}
\usepackage[alphabetic]{amsrefs}
\usepackage{amssymb}
\usepackage{graphicx}
\usepackage{eucal}

\newcommand{\R}{\mathbb{R}}
\newcommand{\N}{\mathbb{N}}
\newcommand{\M}{\mathcal{M}}
\renewcommand{\H}{\mathcal{H}}
\newcommand{\K}{\mathfrak{K}}
\renewcommand{\S}{\mathfrak{S}}
\renewcommand{\P}{\mathfrak{P}}
\newcommand{\Q}{\mathfrak{Q}}
\newcommand{\D}{\mathcal{D}}
\newcommand{\loc}{\mathrm{loc}}
\newcommand{\supp}{\operatorname{supp}}
\renewcommand{\Re}{\operatorname{Re}}
\renewcommand{\Im}{\operatorname{Im}}
\def\Xint#1{\mathchoice 
{\XXint\displaystyle\textstyle{#1}}%
{\XXint\textstyle\scriptstyle{#1}}%
{\XXint\scriptstyle\scriptscriptstyle{#1}}%
{\XXint\scriptscriptstyle\scriptscriptstyle{#1}}%
\!\int} 
\def\XXint#1#2#3{{\setbox0=\hbox{$#1{#2#3}{\int}$} 
\vcenter{\hbox{$#2#3$}}\kern-0.5\wd0}} 
 
\def\dashint{\Xint-}

\renewcommand{\phi}{\varphi}
\renewcommand{\epsilon}{\varepsilon}

\newtheorem{theorem}{Theorem}[section]
\newtheorem{lemma}[theorem]{Lemma}
\newtheorem{proposition}[theorem]{Proposition}
\theoremstyle{definition}
\newtheorem{definition}[theorem]{Definition}
\theoremstyle{remark}
\newtheorem{remark}[theorem]{Remark}
\numberwithin{equation}{section}
\numberwithin{section}{part}

\title[Monotonicity Formulas and Singular Set]{Some New Monotonicity
  Formulas \\ and the Singular
  Set \\ in the Lower Dimensional Obstacle Problem}

\author{Nicola Garofalo}
\address{Department of Mathematics, Purdue University, West Lafayette,
  IN 47907}
\email{garofalo@math.purdue.edu}
\thanks{N.~Garofalo is supported in part by NSF grant DMS-0701001}

\author{Arshak Petrosyan}
\address{Department of Mathematics, Purdue University, West Lafayette,
  IN 47907}
\email{arshak@math.purdue.edu}
\thanks{A.~Petrosyan is supported in part by  NSF grant DMS-0701015}

\subjclass[2000]{35R35}
\keywords{Thin obstacle problem, Signorini problem, singular points,
  Almgren's frequency formula, Weiss's montonicity formula, Monneau's
  monotonicity formula}

\begin{document}
\begin{abstract} We construct two new one-parameter families of monotonicity
formulas to study the free boundary points in the lower dimensional
obstacle problem. The first one is a family of Weiss type formulas
geared for points of any given homogeneity and the second one is a
family of Monneau type formulas suited for the study of singular
points. We show the uniqueness and continuous dependence of the
blowups at singular points of given homogeneity. This allows to
prove a structural theorem for the singular set.

Our approach works both for zero and smooth non-zero lower
dimensional obstacles. The study in the latter case is based on a
generalization of Almgren's frequency formula, first established by
Caffarelli, Salsa, and Silvestre.
\end{abstract}

\maketitle

\section*{Introduction}

\subsection*{The lower dimensional obstacle problem}

Let $\Omega$ be a domain in $\R^n$ and $\M$ a smooth
$(n-1)$-dimensional manifold in $\R^n$ that divides $\Omega$ into two
parts: $\Omega_+$ and $\Omega_-$. For given functions $\phi:\M\to \R$
and $g:\partial\Omega\to\R$ satisfying $g>\phi$ on $\M\cap\partial\Omega$, consider the problem of minimizing the Dirichlet integral
$$
D_\Omega(u)=\int_\Omega |\nabla u|^2 dx
$$
on the closed convex set
$$
\K=\{u\in W^{1,2}(\Omega) \mid \text{$u=g$ on $\partial\Omega$, $u\geq
  \phi$ on $\M\cap\Omega$}\}.
$$
This problem is known as the lower dimensional, or \emph{thin
obstacle problem}, with $\phi$ known as the \emph{thin obstacle}. It
is akin to the \emph{classical obstacle problem}, where $u$ is
constrained to stay above an obstacle $\phi$ which is assigned in
the whole domain $\Omega$. However, whereas the latter is by now
well-understood, the thin obstacle problem still presents
considerable challenges and only recently there has been some
significant progress on it. While we defer  a discussion of the
present status of the problem and of some important open questions
to the last section of this paper, for an introduction to this and
related problems we refer the reader to the book by Friedman
\cite{Fr}*{Chapter 1, Section 11} as well as to the survey of
Ural'tseva \cite{Ur}.

The thin obstacle problem arises in a variety of situations of
interest for the applied sciences. It presents itself in elasticity
(see for instance \cite{KO}), when an elastic body is at rest,
partially laying on a surface $\M$. It also arises in financial
mathematics in situations in which the random variation of an
underlying asset changes discontinuously, see \cite{CT}, \cite{Si}
and the references therein. It models the flow of a saline
concentration through a semipermeable membrane when the flow occurs
in a preferred direction (see \cite{DL}).

When $\M$ and $\phi$ are smooth, it has been proved by Caffarelli
\cite{Ca} that the minimizer $u$ in the thin obstacle problem is of
class $C^{1,\alpha}_\loc(\Omega_\pm\cup\M)$. Since we can make free
perturbations away from $\M$, it is easy to see that $u$ satisfies
$$
\Delta u=0\quad\text{in }\Omega\setminus\M=\Omega_+\cup\Omega_-,
$$
but in general $u$ does not need to be harmonic across $\M$.
Instead, on $\M$, one has the following complementary conditions
$$
u-\phi\geq 0,\quad \partial_{\nu^+}u+\partial_{\nu^-}u\geq 0,\quad
(u-\phi)(\partial_{\nu^+}u+\partial_{\nu^-}u)=0,
$$
where $\nu^\pm$ are the outer unit normals to $\Omega_\pm$ on $\M$.
One of the main objects of study in this problem is the so-called
\emph{coincidence set}
$$
\Lambda(u):=\{x\in\M \mid u(x)=\phi(x)\}
$$
and its boundary (in the relative topology on $\M$)
$$
\Gamma(u):=\partial_{\M}\Lambda(u),
$$
known as the \emph{free boundary}.

A similar problem is obtained when $\M$ is a part of $\partial \Omega$
and one minimizes $D_\Omega(u)$ over the convex set
$$
\K=\{u\in W^{1,2}(\Omega) \mid \text{$u=g$ on
  $\partial\Omega\setminus\M$, $u\geq \phi$ on $\M$}\}.
$$
In this case $u$ is harmonic in $\Omega$ and satisfies the
complementary conditions
$$
u-\phi\geq 0,\quad \partial_\nu u\geq 0,\quad (u-\phi)\partial_\nu
u=0
$$
on $\M$, where $\nu$ is the outer unit normal on $\partial \Omega$.
This problem is known as the \emph{boundary thin obstacle
  problem} or the \emph{Signorini problem}. Note that in the case when
$\M$ is a plane and $\Omega$ and $g$ are symmetric with respect to
$\M$, then the thin obstacle problem in $\Omega$ is equivalent to the boundary
obstacle problem in $\Omega_+$.

\subsection*{Recent developments} There has been a recent surge of
activity in the area of thin obstacle problems since the work of
Athanasopoulos and Caffarelli \cite{AC}, where the optimal
$C^{1,\frac12}$ interior regularity has been established for the
solutions of the Signorini problem with flat $\M$  and $\phi=0$. A
different perspective was brought in with the paper \cite{ACS},
where extensive use was made of the celebrated monotonicity of
Almgren's \emph{frequency function},
$$
N(r,u):=\frac{r\int_{B_r}|\nabla u|^2}{\int_{\partial B_r} u^2},
$$
see \cite{Al}, and also \cite{GL}, \cite{GL2} for generalizations to
variable coefficient elliptic operators in divergence form. The name
frequency comes from the fact that when $u$ is a harmonic function
in $B_1$ homogeneous of degree $\kappa$, then $N(r,u) \equiv
\kappa$. Using such monotonicity of the frequency the authors were
able to show fine regularity properties of the free boundary;
namely, that the set of so-called regular free boundary points is
locally a $C^1$-manifold of dimension $n-2$. When the obstacle is
the function $\phi \equiv 0$ and $\M$ is flat, then a point of the
free boundary is called \emph{regular} if at such point the
frequency attains its least possible value $N(0+,u) = 2 - \frac 12$
(see Lemma \ref{lem:min-homogen} and Definition \ref{def:regFB}
below). For instance the origin is a regular free boundary point for
$\hat u_{3/2}(x) = \Re (x_1 + i\,|x_n|)^{3/2}$. The reader is also referred to
Section~\ref{sec:known-results} for a detailed description of the
main results in \cite{AC} and \cite{ACS}.

Another line of developments has emerged from the fact that in the
particular case
$\Omega_+=\R^n_+:=\R^{n-1}\times(0,\infty)$ and
$\M=\R^{n-1}\times\{0\}$, the Signorini problem can be interpreted as
an obstacle problem for the fractional Laplacian on $\R^{n-1}$
$$
u-\phi\geq 0,\quad (-\Delta_{x'})^{s} u\geq 0,\quad (u-\phi)
(-\Delta_{x'})^{s} u=0
$$
with $s=\frac12$. A more general range of fractional powers $0<s<1$
has been considered by Silvestre \cite{Si}, who has proved the
almost optimal regularity of the solutions: $u\in
C^{1,\alpha}(\R^{n-1})$ for any $\alpha<s$. Subsequently,
Caffarelli, Salsa, and Silvestre \cite{CSS} have established the
optimal $C^{1,s}$ regularity and generalized the free boundary
regularity results of \cite{ACS} to this setting.

One interesting aspect of \cite{CSS} is that the authors are able to
treat the case when the thin obstacle $\phi$ is nonzero. They prove
a generalization of Almgren's monotonicity of the frequency for
solutions of the thin obstacle problem with nonzero thin obstacles,
see Section~\ref{sec:gener-freq-form} below. With this tool they
establish the optimal interior regularity of the solution, and the
regularity of the free boundary at the regular points.

\subsection*{Main results}
In the thin obstacle problem one can subdivide the free boundary
points into three categories: the set of regular points discussed
above, the set $\Sigma(u)$ of the so-called singular points, and the
remaining portion of those free boundary points which are neither
regular, nor singular. As we have mentioned, the papers \cite{ACS},
\cite{Si} and \cite{CSS} study the former set.

The main objective of this paper is the study of the singular free
boundary points. More specifically, using methods from geometric
PDE's we study the structure of the singular set $\Sigma(u)$, which
we now define.

Hereafter, the hypersurface $\M$ is assumed to be the hyperplane
$\{x_n=0\}$. A free boundary point $x_0\in \Gamma(u)$ is called
\emph{singular} if the coincidence set $\Lambda(u)$ has a vanishing
$(n-1)$-Hausdorff density at $x_0$, i.e.
$$
\frac{\H^{n-1}(\Lambda(u)\cap B'_r(x_0))}{\H^{n-1}(B_{r}'(x_0))}\to 0.
$$
We denote by $\Sigma(u)$ the set of singular points. We observe here
that $\Sigma(u)$ is not necessarily a small part of the free
boundary $\Gamma(u)$ in any sense. In fact, it may happen that the
whole free boundary is composed exclusively of singular points. This
happens for instance when $u$ is a harmonic function, symmetric with
respect to $\M$, touching a zero obstacle (see also
Figure~\ref{fig:example} below).

One of the main difficulties in our analysis consists in
establishing the uniqueness of the blowups, which are the limits of
properly defined rescalings of $u$, see \eqref{eq:rescaling} below.
Proving such uniqueness is equivalent to showing that at any $x_0\in
\Sigma(u)$ one has a Taylor expansion
$$
u(x',x_n)-\phi(x')=p_\kappa^{x_0}(x-x_0)+o(|x-x_0|^\kappa),
$$
where $p=p_\kappa^{x_0}$ is a nondegenerate homogeneous polynomial of
a certain order
$\kappa$, satisfying
$$
\Delta p=0,\quad x\cdot\nabla p -\kappa p=0,\quad p(x',0)\geq0,\quad p(x',-x_n)=p(x',x_n),
$$
see Theorems~\ref{thm:k-diff-sing-p} and
\ref{thm:k-diff-sing-p-nonzero} below. The value of $\kappa$ must be
an even integer, and it is obtained from the above mentioned
generalization of Almgren's monotonicity of frequency to the thin
obstacle problem, see Theorems~\ref{thm:almgren} and
\ref{thm:gen-almgren}.

At this point, to put our discussion in a broader perspective, we
recall that for the classical obstacle problem a related Taylor
expansion was originally obtained by Caffarelli and Rivi\`{e}re \cite{CR} in dimension 2. In higher dimensions this expansion was first proved by Caffarelli \cite{Ca2} with the
use of a deep monotonicity formula of Alt, Caffarelli, and Friedman
\cite{ACF}. Subsequently, a different proof based on a simpler
monotonicity formula was discovered by Weiss \cite{We}. More
recently, using the result of Weiss, Monneau \cite{Mo} has derived
yet another monotonicity formula which is tailor made for the study
of the singular free boundary points. He has then used
such formula to prove the above mentioned Taylor expansion at
singular points of the classical obstacle problem.

Now, in the classical obstacle problem the only frequency that
appears is $\kappa =2$.  Specifically, the above mentioned
monotonicity formulas in \cite{ACF}, \cite{We}, \cite{Mo} are only
suitable for $\kappa=2$. In the thin obstacle problem, instead, one
the main complications is that at a singular free boundary point the
frequency $\kappa$ may be an arbitrary even integer $2m$, $m\in\N$.

With this observation in mind, and the objective of studying
singular points, our original desire was to construct an analogue of
Monneau's formula based on Almgren's frequency formula, rather than
on Weiss'. This was suggested by the fact that, at least in
principle,  Almgren's frequency formula does not display the
limitation of the specific value $\kappa = 2$. In the process,
however, we have discovered a new one-parameter family of
monotonicity formulas $\{W_\kappa\}$ of Weiss type (see Theorem
\ref{thm:weiss}) which is tailor made for studying the thin obstacle
problem, and that, remarkably, is inextricably connected to
Almgren's monotonicity formula, see
Section~\ref{sec:weiss-type-monot} and \ref{sec:weiss-monneau-type}.
With these new formulas in hand, following Monneau \cite{Mo} we have
discovered another one-parameter family $\{M_\kappa\}$ of
monotonicity formulas (see Theorem \ref{thm:Monn-mon-form}) which
are ad hoc for studying singular free boundary points with frequency
$\kappa=2m$, $m\in\N$. With this result, in turn, we have been able
to establish the desired Taylor expansion mentioned above, thus
obtaining the uniqueness of the blowups. Furthermore, the
monotonicity formulas $\{M_\kappa\}$ allow to establish the
nondegeneracy and continuous dependence of the polynomial
$p_\kappa^{x_0}$ on the singular free boundary point $x_0$ with
frequency $\kappa$.

We should also mention here that in the case of the nonzero thin
obstacle $\phi$ there are additional technical difficulties
introduced by the error terms in the computations. In fact,
Almgren's monotonicity formula in its purest form will not hold in
general, but if $\phi$ is assumed to be $C^{k,1}$ regular we can
establish the monotonicity of the truncated versions $\Phi_k$ of the
frequency functional, see Theorem~\ref{thm:gen-almgren}. This kind
of formula has been used first in \cite{CSS} in the case $k=2$.
Because of the truncation, our approach allows to effectively study
only the free boundary points at which the frequency takes a value
$\kappa<k$.

Finally, a standard argument based on
Whitney's extension theorem implies that the set of singular points is
contained in a countable union of $C^1$ regular manifolds of
dimensions $d=0,1,\ldots, n-2$. In the case of a nonzero obstacle
$\phi\in C^{k,1}$ this result is limited to singular points with
the frequency $\kappa<k$. For a precise formulation, see
Theorems~\ref{thm:sing-points} and \ref{thm:sing-points-nonzero}.

\subsection*{Structure of the paper} When the thin
obstacle $\phi$ is identically zero our constructions and proofs are
most transparent. While the majority of our results continue to hold
for regular nonzero obstacles, the technicalities of the proofs are
overwhelming and may easily distract from the main ideas. For this
reason we have subdivided the paper into two parts:
Part~\ref{part:zero-thin-obstacle} deals exclusively with solutions
with a zero thin obstacle $\phi$, whereas
Part~\ref{part:nonz-thin-obst} deals with a nonzero $\phi$. The
individual structure of these parts is as follows.

\medskip
\vbox{Part~\ref{part:zero-thin-obstacle}: $\phi=0$.
\begin{itemize}
\item In Section~\ref{sec:normalization}
we define the class $\S$ of normalized solutions of the Signorini problem.
\item  In Section~\ref{sec:known-results} we describe the known results,
   including the optimal regularity and the regularity of the free boundary.
\item Section~\ref{sec:singular-set:-main} contains the statements of our
   main results.

\item In Section~\ref{sec:weiss-type-monot} we establish the above mentioned Weiss and
   Monneau type monotonicity formulas.

\item Finally, in Section~\ref{sec:singular-points} we use these
   monotonicity formulas to establish the structure of the singular
   set.
 \end{itemize}
}
\vbox{
Part~\ref{part:nonz-thin-obst}: $\phi\not=0$.
\begin{itemize}
\item In Section~\ref{sec:normalization-1} we describe a method based on
  harmonic extension of the $k$-th Taylor's polynomial of the thin
  obstacle to obtain the main class $\S_k$ of normalized solutions.

\item In Section~\ref{sec:gener-freq-form} we prove a
   form of Almgren's monotonicity formula which generalizes a similar result in \cite{CSS}.

\item In Section~\ref{sec:growth-near-free} using this monotonicity of the generalized frequency we study the growth of
   $u\in\S_k$ near the origin.

\item In Section~\ref{sec:blowups} we establish the existence of blowups.

\item In Section~\ref{sec:free-boundary} we give a classification of free
   boundary points for solutions of the Signorini problem.

\item In Section~\ref{sec:singular-set:-main-1} we state our main
  results, see Theorems~\ref{thm:k-diff-sing-p-nonzero} and
  \ref{thm:sing-points-nonzero} below.

\item In Section~\ref{sec:weiss-monneau-type} we prove two
extended forms  of the Weiss and the Monneau type monotonicity
formulas obtained in the first part of the paper.

\item Finally, in Section~\ref{sec:singular-set:-proofs}, we prove our
   main results.
 \end{itemize}
 }

In closing, we would like to mention that, using the extension
approach developed in \cite{CS} and \cite{CSS}, our technique works
also for solutions of the obstacle problem for the fractional
Laplacian $(-\Delta_{x'})^s$ for any $0<s<1$. However, for the sake
of exposition, here we restrict ourselves to the case $s=1/2$ (i.e.\
the Signorini problem). The consideration of the general case is
deferred to a forthcoming paper.

\subsection*{Acknowledgement} We would like to thank the referee for his/her insightful comments which have contributed to improving the presentation of the paper. It has also led us to provide an affirmative answer to a question which appeared as a conjecture in the first draft of this paper (see the part $\Sigma_\kappa=\Gamma_\kappa$ in Theorems~\ref{thm:sing-points} and \ref{thm:sing-points-nonzero}).

\newpage
\part{Zero Thin Obstacle}
\label{part:zero-thin-obstacle}

\section{Normalization}
\label{sec:normalization}
In this part of the paper we
consider a solution $u$ of the Signorini problem with \emph{zero
  obstacle} on a \emph{flat boundary}, i.e.\ $\phi=0$ and $\M$ a
hyperplane. Since we are interested in the properties of $u$ near a
free boundary point, after possibly a translation, rotation and
scaling, we can assume that $u$ is defined in $B_1^+ \cup B_1'$,
where
$$B_1^+:=B_1\cap \R^n_+,\quad B_1':=B_1\cap
(\R^{n-1}\times\{0\}).
$$
Moreover, $u\in C^{1,\alpha}_{\loc}(B_1^+\cup B_1')$, and it is such
that
\begin{gather}
\label{eq:signorini-1}
\Delta u=0\quad\text{in } B_1^+\\
\label{eq:signorini-2}
u\geq 0,\quad -\partial_{x_n} u\geq 0,\quad u\,\partial_{x_n} u=0\quad
\text{on }B_1'\\
\label{eq:0-fbp} 0\in\Gamma(u):= \partial \Lambda(u) :=
\partial\{(x',0)\in B_1' \mid u(x',0)=0\},
\end{gather}
where, we recall, $\Lambda(u)\subset B_1'$ is the coincidence set,
and the free boundary $\Gamma(u)$ is the topological boundary of
$\Lambda(u)$ in the relative topology of $B_1'$.
\begin{definition}  Throughout the paper we denote by $\S$ the class of solutions of the
  normalized Signorini problem
  \eqref{eq:signorini-1}--\eqref{eq:0-fbp}.
\end{definition}
Note that we may actually extend $u\in\S$ by even symmetry to $B_1$
\begin{equation}\label{eq:sym}
u(x',-x_n):=u(x',x_n).
\end{equation}
Then the resulting function will satisfy
\begin{alignat*}{2}
  \Delta u&\leq 0&\quad&\text{in } B_1\\
  \Delta u&=0&&\text{in }B_1\setminus\Lambda(u)\\
  u\,\Delta u&= 0&&\text{in } B_1.
\end{alignat*}
It will also be useful to note the following direct relation between
$\Delta u$ and $\partial_{x_n} u$:
$$
\Delta u=2(\partial_{x_n}u)\, \H^{n-1}\big
|_{\Lambda(u)}\quad\text{in }\D'(B_1).
$$
Note that here by $\partial_{x_n} u$ on $\{x_n=0\}$ we understand the limit $\partial_{x_n^+} u$ from inside $\{x_n>0\}$. We keep this convention throughout the paper.


\section{Known Results}
\label{sec:known-results}

\subsection{Optimal regularity} For the solutions of the Signorini problem, finer regularity results are known. It
has been proved by Athanasopoulos and Caffarelli \cite{AC} that in fact $u\in
C^{1,\frac12}_\loc$ in $B_1^+\cup B'_1$. This is the optimal regularity as
one can see from the explicit example of a solution $\hat u_{3/2}$ given
by
\begin{equation}\label{eq:reg-glob-sol}
\hat u_{3/2}(x)=\Re (x_1+i\,|x_n|)^{3/2}.
\end{equation}
Below we indicate the main steps in the proof of this optimal
regularity, by following the approach developed by Athanasopoulos,
Caffarelli, and Salsa \cite{ACS} and Caffarelli, Salsa, and
Silvestre \cite{CSS}. The main analysis is performed by considering
the \emph{rescalings}
\begin{equation}\label{eq:rescaling}
u_r(x):=\frac{u(rx)}{\left(\frac1{r^{n-1}}\int_{\partial B_r} u^2\right)^{1/2}},
\end{equation}
and studying the limits as $r\to 0+$, known as the \emph{blowups}.
We emphasize that from the definition \eqref{eq:rescaling} one has
\begin{equation}\label{eq:one} \|u_r\|_{L^2(\partial B_1)}=1.
\end{equation}
Note that generally the blowups might be different over different
subsequences $r=r_j\to 0+$. The following monotonicity formula plays
a fundamental role in controlling the rescalings.

\begin{theorem}[Monotonicity of the Frequency]
\label{thm:almgren} Let $u$ be a nonzero solution of
\text{\eqref{eq:signorini-1}--\eqref{eq:signorini-2}}, then the frequency of
$u$
$$
r \mapsto N(r,u):=\frac{r\int_{B_r}|\nabla u|^2}{\int_{\partial B_r}
u^2}
$$
is nondecreasing for $0<r<1$. Moreover, $N(r,u)\equiv \kappa$ for
$0<r<1$ if and only if $u$ is homogeneous of degree $\kappa$ in
$B_1$, i.e.
$$
x\cdot\nabla u-\kappa u=0\quad\text{in }B_1.
$$
\end{theorem}
In the case of the harmonic functions this is a classical result of
Almgren \cite{Al}, which was subsequently generalized to divergence
form elliptic operators with Lipschitz coefficients in \cite{GL},
\cite{GL2}. For the thin obstacle problem this formula has been
first used in \cite{ACS}. We will provide a proof of
Theorem~\ref{thm:almgren} in Section~\ref{sec:weiss-type-monot}. The
reason for doing it is twofold. Besides an obvious consideration of
completeness, more importantly we will prove that
Theorem~\ref{thm:almgren} is in essence equivalent to a new
one-parameter family of monotonicity formulas similar to that of Weiss
in \cite{We}, see Theorem~\ref{thm:weiss}.

The following property of the frequency plays an important role: for
any $0<r,\rho<1$ one has
\begin{equation}\label{eq:scalefree}
N(\rho,u_r) = N(r\rho,u).
\end{equation}

\medskip Suppose now $u\in\S$ and $0\in\Gamma(u)$. Consider the
rescalings $u_r$ as defined in \eqref{eq:rescaling}. Using
\eqref{eq:one}, \eqref{eq:scalefree} and the monotonicity of the
frequency $N$ claimed in Theorem~\ref{thm:almgren}, one easily has
for $r\leq 1$
$$
\int_{B_1}|\nabla u_r|^2=N(1,u_r)=N(r,u)\leq N(1,u).
$$
Now, this implies that there exists a nonzero function $u_0\in
W^{1,2}(B_1)$, which we call a \emph{blowup} of $u$ at the origin,
such that for a subsequence $r=r_j\to 0+$
\begin{equation}\label{eq:urj-u0-conv}
\begin{aligned}
  u_{r_j}\to u_0 &\quad\text{in }W^{1,2}(B_1)\\
  u_{r_j}\to u_0 &\quad\text{in }L^2(\partial B_1)\\
  u_{r_j}\to u_0 &\quad\text{in }C^1_\loc(B^\pm_1\cup B_1').
\end{aligned}
\end{equation}
It is easy to see the weak convergence in $W^{1,2}(B_1)$ and the
strong convergence in $L^2(\partial B_1)$. The third convergence
(and consequently the strong convergence in $W^{1,2}$) follows from
uniform $C^{1,\alpha}_\loc$ estimates on $u_r$ in $B_1^\pm\cup B_1'$
in terms of $W^{1,2}$-norm of $u_r$ in $B_1$, see e.g.\ \cite{AC}.

\begin{proposition}[Homogeneity of blowups]\label{prop:blowup-homogen}
Let $u\in\S$ and denote by $u_0$ any
  blowup of $u$ as described above. Then  $u_0$ satisfies \text{\eqref{eq:signorini-1}--\eqref{eq:signorini-2}}, is homogeneous of
  degree $\kappa=N(0+,u)$, and $u_0\not\equiv 0$.
\end{proposition}

\begin{proof}
The fact that $u_0$ satisfies
\eqref{eq:signorini-1}--\eqref{eq:signorini-2} follows from the
above mentioned $C^{1,\alpha}_\loc$ estimates on $u_r$ in
$B_1^\pm\cup B_1'$.
  For the blowup $u_0$ over a sequence $r_j\to 0+$ we have
  $$
  N(r,u_0)=\lim_{r_j\to 0+}N(r,u_{r_j})=\lim_{r_j\to 0+} N(r r_j,
  u)= N(0+,u)
  $$
  for any $0<r<1$. This implies that $N(r, u_0)$ is a constant. In
  view of the last part of Theorem~\ref{thm:almgren} we
conclude that $u_0$ is a homogeneous function. The fact that
  $u_0\not\equiv 0$ follows from the convergence $u_{r_j}\to
u_0$ in $L^2(\partial B_1)$ and that equality $\int_{\partial B_1}
u_{r_j}^2=1$, implying that $\int_{\partial B_1} u_{0}^2=1$.
\end{proof}

We emphasize that although the blowups at the origin might not be
unique, as a consequence of Proposition~\ref{prop:blowup-homogen}
they all have the same homogeneity.

\begin{lemma}[Minimal homogeneity]\label{lem:min-homogen} Given $u\in\S$
one has
  $$
  N(0+,u)\geq 2-\frac12.
  $$
  Moreover, either
  $$
  N(0+,u)=2-\frac12\quad\text{or}\quad N(0+,u)\geq 2.\qed
  $$
\end{lemma}

For the proof see \cite{CSS}*{Lemma 6.1}.  This follows from the
classification of the homogeneous solutions of the Signorini problem
which are convex in the $x'$-variables. The lower bound is
essentially contained in Silvestre's dissertation \cite{Si}. The
last part of the lemma was first proved in \cite{ACS}.

\medskip
The minimal homogeneity allows to establish the following maximal
growth of the solution near free boundary points, see
\cite{CSS}*{Theorem~6.7}.

\begin{lemma}[Growth estimate] Let $u\in\S$. Then
  $$
  \sup_{B_r} |u|\leq C\, r^{3/2},\quad 0<r<1/2,
  $$
  where $C=C(n,\|u\|_{L^2(B_1)})$.\qed
\end{lemma}

Ultimately, this leads to the optimal regularity of $u$.

\begin{theorem}[Optimal regularity]
\label{thm:classification-fbp}
Let $u\in\S$ and $0\in \Gamma(u)$. Then $u\in
C^{1,\frac12}_{\loc}(B_{1}^\pm\cup B_1')$ with
$$
\|u\|_{C^{1,\frac12}(B_{1/2}^{\pm}\cup B_{1/2}')}\leq
C(n,\|u\|_{L^2(B_1)}).
$$
\qed
\end{theorem}

\subsection{Regularity of the free boundary}

Another aspect of the Signorini problem is the study of the free
boundary $\Gamma(u)$. In fact, the starting point in the study of
the regularity of the free boundary  is precisely the optimal
regularity of $u$, which we have described in Theorem
\ref{thm:classification-fbp}. First note that Almgren's frequency
functional (as well as
Theorems~\ref{thm:almgren}--\ref{thm:classification-fbp}) can be
defined at any point $x_0\in \Gamma(u)$ by simply translating that
point to the
 origin: 
 $$
 N^{x_0}(r,u):=\frac{r\int_{B_r(x_0)}|\nabla u|^2}{\int_{\partial
     B_r(x_0)} u^2},
 $$
 where $r>0$ is such that $B_r(x_0)\Subset B_1$. This
 enables us to give the following definitions.

\begin{definition}\label{def:fb-class} Given $u\in\S$, for $\kappa\geq 2-\frac12$ we define
  $$
  \Gamma_\kappa(u):=\{x_0\in\Gamma(u) \mid N^{x_0}(0+,u)=\kappa\}.
  $$
\end{definition}

\begin{figure}[tbp]
\begin{picture}(318,174)(0,0)
  \includegraphics[width=144pt]{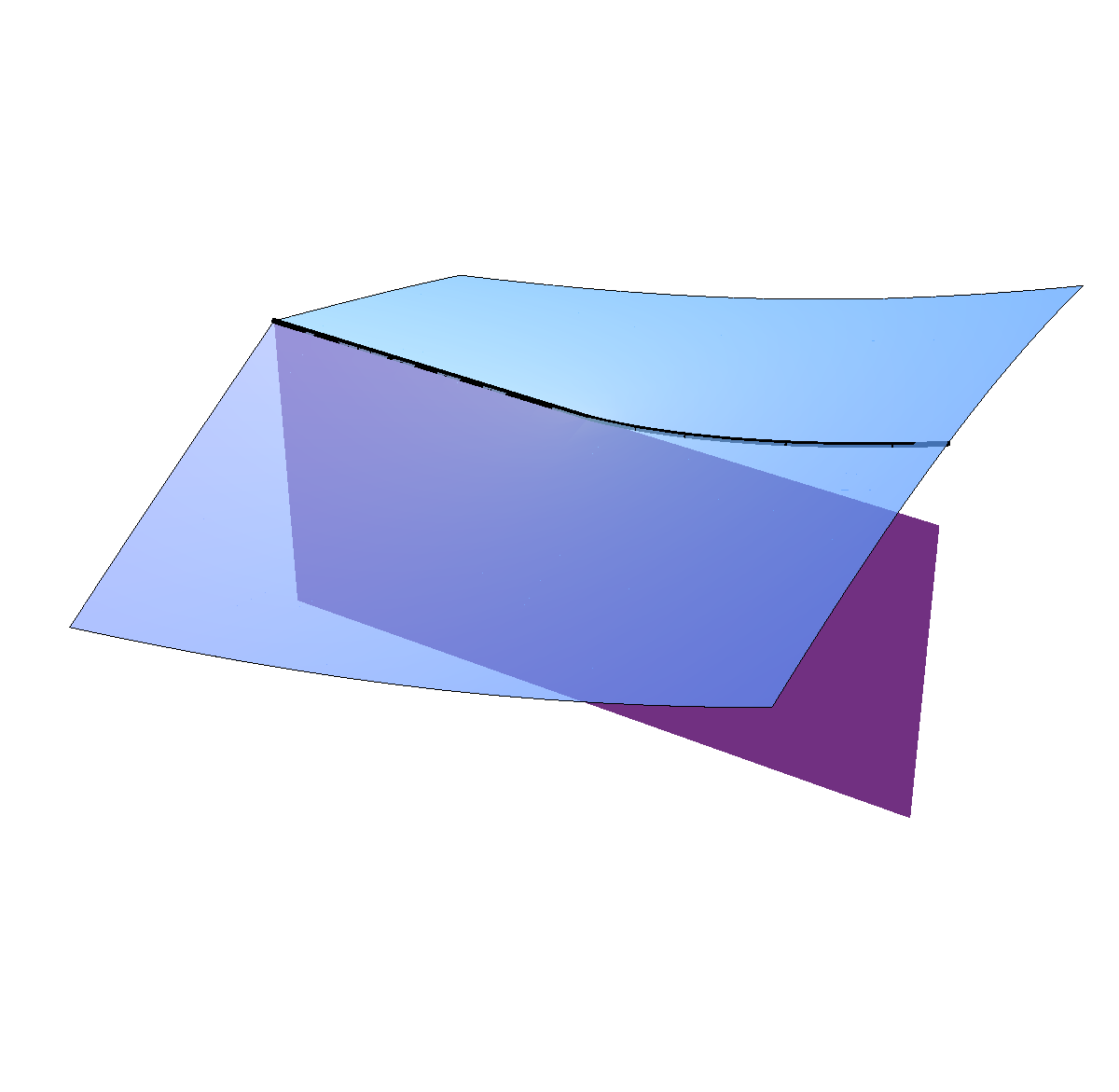}
  \hskip30pt
  \includegraphics[width=144pt]{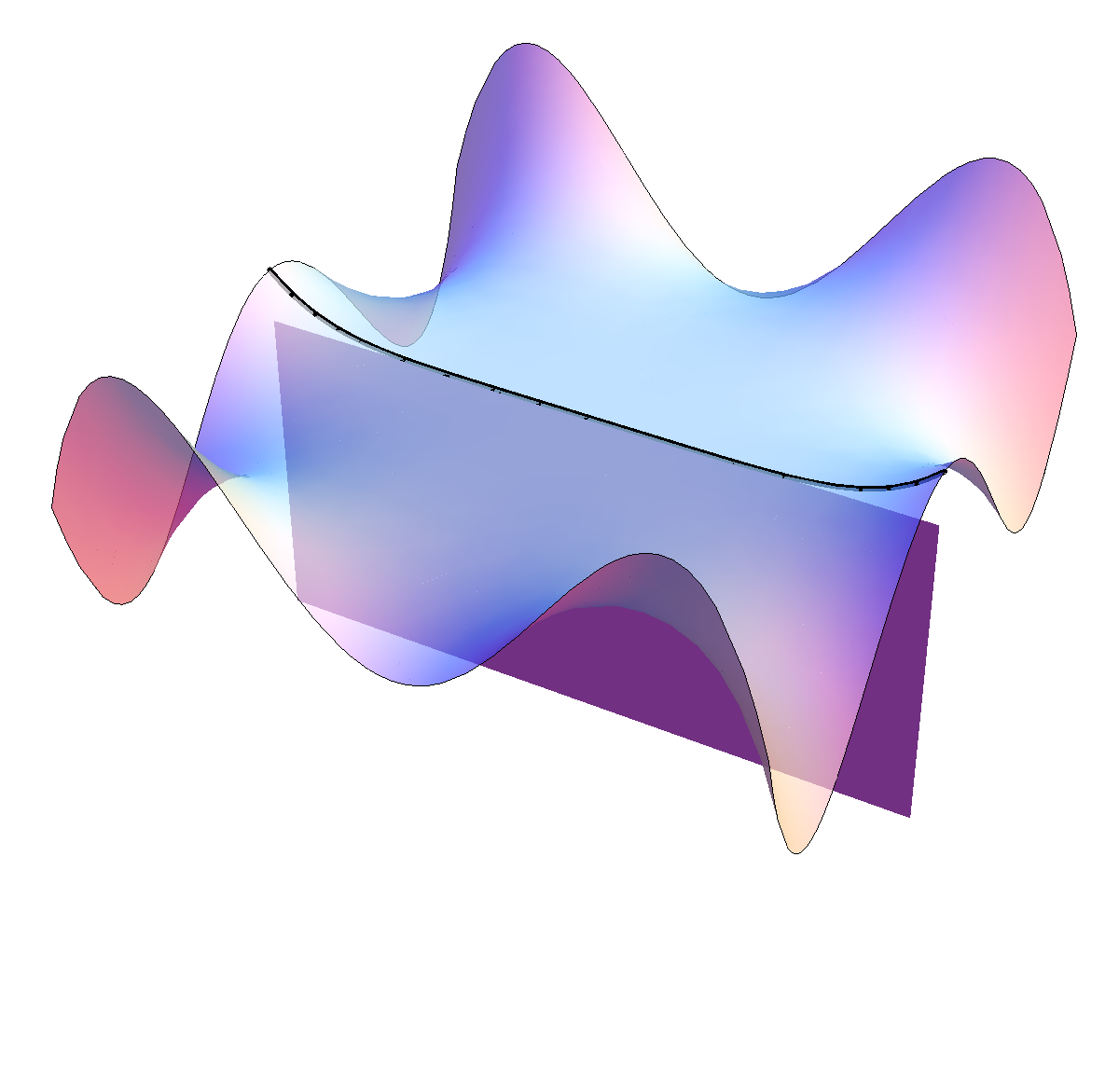}
\end{picture}
\caption{Graphs of $\Re(x_1+i\,|x_2|)^{3/2}$ and $\Re(x_1+i\,|x_2|)^{6}$}
\label{fig:label}
\end{figure}

\begin{remark}\label{rem:missing_values}
One has to point out that the sets $\Gamma_\kappa(u)$ may be
nonempty only for $\kappa$ in a certain set of values. For instance,
Lemma~\ref{lem:min-homogen} implies that $\Gamma_k(u)=\varnothing$
whenever $2-\frac12<\kappa<2$. On the other hand, if one considers
the functions
$$
\hat u_\kappa(x) = \Re (x_1+i\,|x_n|)^{\kappa},\quad\text{for }\kappa\in
\{2m-\frac12 \mid m\in \mathbb N\}\cup\{2m \mid m\in \mathbb N\},
$$
then one has $0\in\Gamma_\kappa(\hat u_\kappa)$, and therefore
$\Gamma_\kappa(\hat u_\kappa)\not=\varnothing$ for any of the above
values of $\kappa$.
\end{remark}

\begin{remark}\label{rem:poss_freq_dim2}
In dimension $n=2$, a simple analysis of homogeneous harmonic
functions in a halfplane shows that, up to a multiple and a mirror reflection, the only
possible solutions of \eqref{eq:signorini-1}--\eqref{eq:signorini-2}
are the functions $\hat u_\kappa$ above and
$$
\hat v_\kappa(x)=\Im (x_1+i\,|x_2|)^\kappa,\quad\text{for }\kappa\in\{2m+1 \mid m\in\N\}.
$$
However, we claim that the values $\kappa\in\{2m+1\mid m\in\N\}$
cannot occur in the blowup for any $u\in\S$. Indeed, since
$0\in\Gamma(u)=\partial\{u(\cdot,0)>0\}$, we may choose a sequence
$r=r_j\to 0+$ so that
$u(\frac12r_j,0)>0$ (or $u(-\frac12r_j,0)>0$). Then
from the complementary condition \eqref{eq:signorini-2} we will have
$\partial_{x_2}u(\frac12r_j,0)=0$ implying that $\partial_{x_2}
u_{r_j}(\frac12,0)=0$. Hence, if $u_0$ is a blowup over a
subsequence of $\{r_j\}$ the $C^1$ convergence will imply that
$\partial_{x_2}u_0(\frac12,0)=0$. However, $\hat v_\kappa$ do not
satisfy this condition.

Thus, the only frequencies $\kappa=N(0+,u)$ that appear in dimension
$n=2$ are $\kappa\in\{2m-\frac12 \mid m\in \mathbb N\}\cup\{2m \mid
m\in \mathbb N\}$. It is plausible that a similar result hold in
higher dimensions, but this is not known to the authors at the time
of this writing. See also our concluding remarks in the last section
of this paper.
\end{remark}

\medskip Of special interest is the case of the smallest possible
value of the frequency $\kappa=2-\frac12$.

\begin{definition}[Regular points]\label{def:regFB} For $u\in\S$ we say that
  $x_0\in\Gamma(u)$ is \emph{regular} if $N^{x_0}(0+,u)=2-\frac12$,
  i.e., if $x_0\in\Gamma_{2-\frac12}(u)$.
\end{definition}

Note that from the Almgren's frequency formula it follows that the
mapping $x_0\mapsto N^{x_0}(0+,u)$ is upper semicontinuous.
Moreover, since $N^{x_0}(0+,u)$ misses values in the interval
$(2-\frac12,2)$, one immediately obtains that
$\Gamma_{2-\frac12}(u)$ is a relatively open subset of $\Gamma(u)$.
The following regularity theorem at regular free boundary points has
been proved by Athanasopoulos, Caffarelli, and Salsa \cite{ACS}.

\begin{theorem}[Regularity of the regular set]\label{thm:regul-points} Let
$u\in\S$, then the free boundary
  $\Gamma_{2-\frac12}(u)$ is locally a $C^{1,\alpha}$ regular
  $(n-2)$-dimensional surface.\qed
\end{theorem}

\section{Singular set: statement of main results}
\label{sec:singular-set:-main}

The main objective of this paper is to study the structure of the
so-called singular set of the free boundary. In this section we
state our main results in this direction, Theorems
\ref{thm:k-diff-sing-p} and \ref{thm:sing-points}. The proofs of
these results will be presented in Section~\ref{sec:singular-points}.

\begin{definition}[Singular points] Let $u\in\S$. We say that $0$ is a
  \emph{singular point} of the free boundary $\Gamma(u)$, if
  $$
  \lim_{r\to 0+}\frac{\H^{n-1}(\Lambda(u)\cap B_r')}{\H^{n-1}(B_r')}=0.
  $$
We denote by $\Sigma(u)$ the subset of singular points of
$\Gamma(u)$. We also denote
\begin{equation}\label{def:sigmakappa}
\Sigma_\kappa(u):=\Sigma(u)\cap\Gamma_\kappa(u).
\end{equation}
\end{definition}
Note that in terms of the rescalings \eqref{eq:rescaling}
the condition $0\in\Sigma(u)$ is equivalent to
$$
\lim_{r\to 0+}\H^{n-1}(\Lambda(u_r)\cap B_1')=0.
$$

The following theorem gives a complete characterization of singular points via the value $\kappa=N(0+,u)$ as well as the type of the blowups. In particular, it establishes that
$$
\Sigma_\kappa(u)=\Gamma_\kappa(u)\quad\text{for }\kappa=2m,\ m\in\N.
$$


\begin{theorem}[Characterization of singular
  points]\label{thm:blowup-sing} Let $u\in\S$ and
  $0\in \Gamma_\kappa(u)$. Then the following statements are equivalent:

\begin{enumerate}
\item[(i)] $0\in \Sigma_\kappa(u)$
\item[(ii)] any blowup of $u$ at the origin is a nonzero homogeneous
 polynomial $p_\kappa$ of degree $\kappa$ satisfying
 $$
 \Delta p_\kappa=0,\quad p_\kappa(x',0)\geq 0,\quad
 p_\kappa(x',-x_n)=p_\kappa(x',x_n).
 $$
\item[(iii)] $\kappa=2m$ for some $m\in\N$.
\end{enumerate} 
\end{theorem}
\begin{proof} (i) $\Rightarrow$ (ii) The rescalings $u_r$ satisfy
  $$
  \Delta u_r=2(\partial_{x_n} u_r)
  \H^{n-1}\big|_{\Lambda(u_r)}\quad\text{in }\D'(B_1).
  $$
  Since $|\nabla u_r|$ are locally uniformly bounded in $B_1$ by
  \eqref{eq:urj-u0-conv}  and 
  $\H^{n-1}(\Lambda(u_r)\cap B_1)\to 0$, the formula above implies that 
  $\Delta u_r$ converges weakly to $0$ in $\D(B_1)$ and therefore any
  blowup $u_0$ must be harmonic in $B_1$. On the other hand, by
  Proposition~\ref{prop:blowup-homogen}, the function $u_0$ is
  homogeneous in $B_1$ and therefore can be extended by homogeneity to
  $\R^n$. The resulting extension will be harmonic in $\R^n$ and, being
  homogeneous, will have at most a polynomial growth at infinity. Then
  by the Liouville theorem we conclude that $u_0$ must be a
  homogeneous harmonic polynomial $p_\kappa$ of a certain integer degree $\kappa$. We
  also have that $p_\kappa\not\equiv 0$ in
  $\R^n$ by Proposition~\ref{prop:blowup-homogen}. The properties of
  $u$ also imply that that 
  $p_\kappa(x',0)\geq 0$ for all $x'\in\R^{n-1}$ 
  and $p_\kappa(x',-x_n)=p_\kappa(x',x_n)$ for all $x=(x',x_n)\in\R^n$.

\medskip\noindent
(ii) $\Rightarrow$ (iii) Let $p_\kappa$ be a blowup of $u$ at the origin. If $\kappa$ is odd, the nonnegativity of $p_\kappa$ on $\R^{n-1}\times\{0\}$
  implies that $p_\kappa$ vanishes on $\R^{n-1}\times\{0\}$
  identically. On the other hand, from the even symmetry we also have
  that $\partial_{x_n} p_\kappa\equiv 0$  on $\R^{n-1}\times\{0\}$. Since
  $p_\kappa$ is harmonic in $\R^n$,  the Cauchy-Kovalevskaya theorem
  implies that $p_\kappa\equiv 0$ in $\R^n$, contrary to the assumption. Thus, $\kappa\in\{2m\mid m\in\N\}$.

\medskip\noindent
(iii) $\Rightarrow$ (ii) The proof is an immediate corollary of the following Liouville type result.

\begin{lemma}\label{lem:2m-homogen} Let $v$ be a $\kappa$-homogeneous global  solution of the thin obstacle problem 
in $\R^n$ with $\kappa=2m$ for $m\in\N$. Then $v$ is a homogeneous harmonic polynomial.
\end{lemma}

Since $\Delta v=2(\partial_{x_n} v) \H^{n-1}\big|_{\Lambda(v)}$ on $\R^n$, with $\partial_{x_n}v\leq 0$ on $\{x_n=0\}$, this is a particular case of the following lemma, which is essentially Lemma~7.6 in Monneau \cite{Mo2}, with an almost identical proof.
 
\begin{lemma}\label{lem:Monn-homogen-harm} Let $v\in W^{1,2}_{\loc}(\R^n)$ satisfy $\Delta v\leq 0$ in $\R^n$ and $\Delta v=0$ in $\R^n\setminus\{x_n=0\}$. If $v$ is homogeneous of degree $\kappa=2m$, $m\in\N$, then $\Delta v=0$ in $\R^n$.
\end{lemma}

\begin{proof} By assumption, $\mu:=\Delta v$ is a nonpositive measure, living on $\{x_n=0\}$. We are going to show that $\mu=0$. To this end, let $P$ be a $2m$-homogeneous harmonic polynomial, which is positive on $\{x_n=0\}\setminus\{0\}$. For instance, take 
$$
P(x)=\sum_{j=1}^{n-1}\Re(x_j+i x_n)^{2m}.
$$	
Further, let $\psi\in C^\infty_0(0,\infty)$ with $\psi\geq 0$ and $\Psi(x)=\psi(|x|)$. Then we have
\begin{align*}
-\langle \mu, \Psi P \rangle &= -\langle \Delta v, \Psi P \rangle=\int_{\R^n} \nabla  v\cdot \nabla (\Psi P)\\
&=\int_{\R^n} \Psi \nabla v\cdot \nabla P + P \nabla v\cdot \nabla \Psi\\
&=\int_{\R^n} -\Psi v \Delta P - v \nabla \Psi\cdot \nabla P+ P \nabla v\cdot \nabla \Psi \\
&=\int_{\R^n} -\Psi v \Delta P-\frac{\psi'(|x|)}{|x|} v (x\cdot \nabla P)+\frac{\psi'(|x|)}{|x|} P(x\cdot \nabla v)=0,
\end{align*}
where in the last step we have used that $\Delta P=0$, $x\cdot\nabla P=2m P$, $x\cdot\nabla v=2m v$.
This implies that the measure $\mu$ is supported at the origin.  Hence $\mu=c \delta_0$, where $\delta_0$ is the Dirac's delta. On the other hand, $\mu$ is $2(m-1)$-homogeneous and $\delta_0$ is $(-n)$-homogeneous and therefore $\mu=0$.
\end{proof}

\noindent
(ii) $\Rightarrow$ (i) Suppose that $0$ is not a singular point and that over some sequence $r=r_j\to 0+$ we have $\H^{n-1}(\Lambda (u_r)\cap B_1')\geq \delta>0$. Taking a subsequence if necessary, we may assume that $u_{r_j}$ converges to a blowup $u_0$. We claim that 
$$
\H^{n-1}(\Lambda (u_0)\cap B_1')\geq \delta>0.
$$
Indeed, otherwise there exists an open set $U$ in $\R^{n-1}$ with $\H^{n-1}(U)<\delta$ so that $\Lambda (u_0)\cap \overline{B_1'} \subset U$. Then for large $j$ we must have $\Lambda (u_{r_j})\cap \overline{B_1'} \subset U$, which is a contradiction, since $\H^{n-1}(\Lambda (u_{r_j})\cap \overline{B_1'})\geq \delta > \H^{n-1}(U)$. But then $u_0$ vanishes identically on $\R^{n-1}\times\{0\}$ and consequently on $\R^n$ by the Cauchy-Kovalevskaya theorem. This completes the proof of the theorem. 
\end{proof}

\begin{definition}
  Throughout the rest of the paper we denote  by $\P_\kappa$, $\kappa=2m$, $m\in \N$, the class of
  $\kappa$-homogeneous harmonic polynomials described in statement (ii) of
  Theorem~\ref{thm:blowup-sing}.
\end{definition}

\begin{theorem}[$\kappa$-differentiability at singular
  points]\label{thm:k-diff-sing-p} Let $u\in\S$ and
  $0\in\Sigma_\kappa(u)$ with $\kappa=2m$, $m\in\N$. Then there exists a
  \emph{nonzero} $p_\kappa\in\P_\kappa$ such that
  $$
  u(x)=p_\kappa(x)+o(|x|^\kappa).
  $$
  Moreover, if for $x_0\in\Sigma_\kappa(u)$ the polynomial
  $p_\kappa^{x_0}\in\P_\kappa$ is such that we have the Taylor
  expansion
  $$
  u(x)=p_\kappa^{x_0}(x-x_0)+o(|x-x_0|^\kappa),
  $$
  then $p_\kappa^{x_0}$ depends continuously on
  $x_0\in\Sigma_\kappa(u)$.
\end{theorem}

We want to point out here that the polynomials $p_\kappa\in\P_\kappa$
can be recovered uniquely from their restriction to $\R^{n-1}\times
\{0\}$. This follows from the Cauchy-Kovalevskaya theorem; see the
proof of the uniqueness part of Lemma~\ref{lem:polynom-ext} in
Part~\ref{part:nonz-thin-obst}. Thus, if $p_\kappa$ is not
identically zero in $\R^n$ then its restriction to
$\R^{n-1}\times\{0\}$ is also nonzero.

Theorem~\ref{thm:k-diff-sing-p} can be used to prove a theorem on the
structure of the singular set, similar to the one of Caffarelli
\cite{Ca2} in the classical obstacle problem. In order to state the
result we define the dimension $d=d_\kappa^{x_0}$ of $\Sigma_\kappa(u)$
at a given point $x_0$ based on the polynomial
$p_\kappa^{x_0}$. Roughly speaking, we expect $\Sigma_\kappa(u)$ to be
contained in a $d$-dimensional manifold near $x_0$.

\begin{definition}[Dimension at the singular
  point]\label{def:dim-sing-point} For a singular point
$x_0\in\Sigma_\kappa(u)$ we denote
$$
d_\kappa^{x_0}:=\dim\{\xi\in\R^{n-1} \mid \xi\cdot \nabla_{x'}
p_\kappa^{x_0}(x',0)=0\text{ for all }x'\in\R^{n-1}\},
$$
which we call the \emph{dimension} of $\Sigma_\kappa(u)$ at $x_0$. Note that
since $p_\kappa^{x_0}\not\equiv 0$ on $\R^{n-1}\times\{0\}$ one has
$$
0\leq d_\kappa^{x_0}\leq n-2.
$$
For $d=0,1,\ldots, n-2$ we define
$$
\Sigma_\kappa^d(u):=\{x_0\in\Sigma_\kappa(u) \mid d^{x_0}_\kappa=d\}.
$$
\end{definition}

\begin{theorem}[Structure of the singular set]\label{thm:sing-points} Let
  $u\in\S$. Then $\Sigma_\kappa(u)=\Gamma_\kappa(u)$ for $\kappa=2m$,
	  $m\in\N$, and every set $\Sigma_\kappa^d(u)$, $d=0,1,\ldots, n-2$ is contained in a countable union of $d$-dimensional $C^1$ manifolds.
\end{theorem}

The following example provides a small illustration of
Theorem~\ref{thm:sing-points}. Consider the harmonic polynomial
$u(x)=x_1^2 x_2^2-\left(x_1^2+x_2^2\right)x_3^2+\frac13{x_3^4}$ in
$\R^3$. Note that $u\in\P_4\subset\S$. On $\R^2\times\{0\}$, we have
$u(x_1,x_2,0)=x_1^2x_2^2$ and therefore the coincidence set
$\Lambda(u)$ as well as the free boundary $\Gamma(u)$ consist of the
union of the lines $\R\times\{0\}\times\{0\}$ and
$\{0\}\times\R\times\{0\}$. Thus, all free boundary points are
singular. It is straightforward to check that $0\in\Sigma_4^0(u)$
and that the rest of the free boundary points are in
$\Sigma_2^1(u)$,  see Figure~\ref{fig:example}.

\begin{figure}[tbp]
\begin{picture}(144,162)(0,0)
  \put(75,76){\small $(0,0)$}
  \put(134,76){\small $x_1$}
  \put(76,134){\small $x_2$}
  \put(47,52){\small $\Sigma_4^0$}
  \put(59,59){\vector(1,1){10}}
  \put(35,76){\small $\Sigma_2^1$}
  \put(76,35){\small $\Sigma_2^1$}
  \put(76,107){\small $\Sigma_2^1$}
  \put(107,76){\small $\Sigma_2^1$}
  \put(72,0){\vector(0,1){144}}
  \put(0,72){\vector(1,0){144}}
  \put(72,72){\circle*{5}}
\end{picture}
\caption{Free boundary for $u(x)=x_1^2 x_2^2-\left(x_1^2+x_2^2\right)
 x_3^2+\frac13{x_3^4}$ in $\R^3$ with zero thin obstacle on
 $\R^2\times\{0\}$.}
\label{fig:example}
\end{figure}

\section{Weiss and Monneau type  monotonicity formulas}
\label{sec:weiss-type-monot}

In this section we introduce two new one-parameter families of
monotonicity formulas that will play a key role in our analysis.
Before doing so, however, we give a proof of Almgren's frequency
formula since the latter has served as one of our main sources of
inspiration. We refer the reader to the original paper by Almgren
\cite{Al} for the case of harmonic functions, to \cite{GL},
\cite{GL2} for solutions to divergence form elliptic equations, and
to Lemma 1 in \cite{ACS} for the thin obstacle problem.

\begin{proof}[Proof of Theorem~\ref{thm:almgren}]
Let $u\in\S$ and consider the quantities
\begin{equation}\label{eq:IH}
D(r):=\int_{B_r} |\nabla u|^2,\qquad H(r):=\int_{\partial
B_r} u^2.
\end{equation}
Denoting by $u_\nu=\partial_\nu u$, where $\nu$ is the outer unit
normal on $\partial B_r$, we have
\begin{equation}\label{eq:H'}
H'(r)=\frac{n-1}{r}\, H(r)+2\int_{\partial B_r} u u_\nu.
\end{equation}
On the other hand, using that $\Delta (u^2/2)=u\Delta u +|\nabla
u|^2=|\nabla u|^2$ and integrating by parts, we obtain
\begin{equation}\label{eq:D-int-parts}
\int_{\partial B_r} uu_\nu =\int_{B_r} |\nabla u|^2=D(r).
\end{equation}
Further, to compute $D'(r)$ we use Rellich's formula
$$
\int_{\partial B_r} |\nabla u|^2=\frac{n-2}{r}\int_{B_r}|\nabla
u|^2+2\int_{\partial B_r} u_\nu^2-\frac2r\int_{B_r} (x\cdot \nabla
u)\Delta u.
$$
Notice that in view of the fact $(x\cdot \nabla u) u_{x_n}=0$ on
$B_1'$  the last integral in the right-hand side vanishes. Hence,
\begin{equation}\label{eq:D'}
D'(r)=\frac{n-2}{r}\,D(r)+2\int_{\partial B_r} u_\nu^2.
\end{equation}
Thus, as in the classical case of harmonic functions we have
\begin{align*}
  \frac{N'(r)}{N(r)}&=\frac{1}{r}+\frac{D'(r)}{D(r)}-\frac{H'(r)}{H(r)}\\
  &=\frac{1}{r}+\frac{n-2}{r}-\frac{n-1}{r}+2\left\{\frac{\int_{\partial
        B_r} u_\nu^2}{\int_{\partial B_r} u
      u_\nu}-\frac{\int_{\partial B_r} u u_\nu}{\int_{\partial B_r}
      u^2}\right\} \geq 0,
\end{align*}
where we have let $N(r) = N(r,u)$. The last inequality is obtained
form the Cauchy-Schwarz inequality and implies the monotonicity
statement in the theorem. Analyzing the case of equality in
Cauchy-Schwarz, we obtain the second part of the theorem. For
details, see the end of the proof of \cite{ACS}*{Lemma 1}.
\end{proof}

\subsection{Weiss type monotonicity formulas}
\label{sec:weiss-type-monot-1} Here we introduce a new one-parameter
family of monotonicity formulas inspired by that introduced by Weiss
\cite{We} in the study of the classical obstacle problem.  Given
$\kappa\geq 0$, we define a functional $W_\kappa(r,u)$, which is
suited for the study of the blowups at free boundary points where
$N(0+,u)=\kappa$.

\begin{theorem}[Weiss type Monotonicity Formula]\label{thm:weiss} Given $u\in\S$, for any $\kappa\geq 0$ we introduce the function
  $$
  W_\kappa(r,u):=\frac{1}{r^{n-2+2\kappa}}\int_{B_r}|\nabla u|^2
  -\frac{\kappa}{r^{n-1+2\kappa}}\int_{\partial B_r} u^2.
  $$
  For $0<r<1$ one has
  $$
  \frac{d}{dr} W_\kappa(r,u)=\frac{2}{r^{n+2\kappa}}\int_{\partial
    B_r}(x \cdot \nabla u- \kappa\,u)^2.
  $$
  As a consequence, $r\mapsto W_\kappa(r,u)$ is nondecreasing on
  $(0,1)$. Furthermore, $W_\kappa(\cdot,u)$ is constant if and
  only if $u$ is homogeneous of degree $\kappa$.
\end{theorem}

\begin{proof} Using the same notations \eqref{eq:IH} as in the proof
  of Theorem~\ref{thm:almgren}, we have
  $$
  W_\kappa(r,u)=\frac{1}{r^{n-2+2\kappa}}\,D(r)-
  \frac{\kappa}{r^{n-1+2\kappa}}\,H(r).
  $$
  Using the identities \eqref{eq:H'}--\eqref{eq:D'}, we
  obtain
\begin{align*}
 \frac{d}{dr} W_\kappa(r,u)&=\frac{1}{r^{n-2+2\kappa}}\left\{D'(r)-\frac{n-2+2\kappa}{r}\,
    D(r)-\frac{\kappa}{r}\,H'(r)
    +\frac{\kappa(n-1+2\kappa)}{r^2}\,H(r)\right\}\\
  &=\frac{1}{r^{n-2+2\kappa}}\left\{2\int_{\partial B_r}
    u_\nu^2-\frac{2\kappa}{r}\int_{\partial B_r} u
    u_\nu-\frac{2\kappa}{r}\int_{\partial B_r}
    uu_\nu+\frac{2\kappa^2}{r^2}\int_{\partial B_r} u^2\right\}\\
  &=\frac{2}{r^{n+2\kappa}}\int_{\partial B_r} (x\cdot\nabla u-\kappa
  u)^2.\qedhere
\end{align*}
\end{proof}
\begin{remark}
We note that in the statement of Theorem~\ref{thm:weiss} it is not necessary to assume $0\in
\Gamma_\kappa(u)$.
However, the monotonicity formula is most useful under such
assumption. The original formula by Weiss \cite{We} for the classical obstacle problem is stated only for $\kappa=2$. While such limitation is natural in the classical obstacle problem, in the lower dimensional obstacle problem we need the full range of $\kappa$'s, or at least $\kappa=2m-\frac12,\ 2m$, $m\in\N$.
\end{remark}

\subsection{Monneau type monotonicity formulas}
\label{sec:monn-type-monot}
The next family of monotonicity formulas is related to a formula
first used by Monneau \cite{Mo} in the study of singular points in the
classical obstacle problem. Here we derive a $\kappa$-homogeneous analogue of
his formula, suited for the study of singular points in the thin
obstacle problem. Recall that for $\kappa=2m$, $m\in\N$ we denote by
$\P_\kappa$ the family of harmonic homogeneous polynomial
$p_\kappa$ of degree $\kappa$, positive on $x_n=0$; i.e.
$$
\P_\kappa=\{p_\kappa(x) \mid \Delta p_\kappa=0,\ x\cdot\nabla
p_\kappa-\kappa p_\kappa=0,\ p_\kappa(x',0)\geq 0.\}
$$

\begin{theorem}[Monneau type Monotonicity
  Formula]\label{thm:Monn-mon-form} Let $u\in \S$ with
  $0\in\Sigma_{\kappa}(u)$, $\kappa=2m$, $m\in\N$. Then for arbitrary
  $p_\kappa\in \P_\kappa$
  $$
  r \mapsto M_\kappa(r,u,p_\kappa):=\frac{1}{r^{n-1+2\kappa}}\int_{\partial
    B_r}(u-p_\kappa)^2
  $$
is nondecreasing for $0<r<1$.
\end{theorem}
\begin{proof} We note that
\[
N(r,u)\geq \kappa ,\quad W_\kappa(r,u)\geq 0,\quad \text{and}\quad W_\kappa(r,p_\kappa)=0.
  \]
The first inequality follows from
  Almgren's monotonicity formula
  $$
  N(r,u)\geq N(0+,u)=\kappa.
  $$
  The second inequality follows from the identity
  $$
  W_\kappa(r,u)=\frac{H(r)}{r^{n-1+2\kappa}}(N(r,u)-\kappa)\geq 0.
  $$
  Finally, the third equality follows from the identity
  $N(r,p_\kappa)=\kappa$.

We then follow the proof of \cite{Mo}*{Theorem 1.8}. Let $w=u-p_\kappa$ and write
\begin{align*}
  W_\kappa(r,u)&=W_\kappa(r,u)-W_\kappa(r,p_\kappa)\\
  &=\frac{1}{r^{n-2+2\kappa}}\int_{B_r} (|\nabla w|^2+2 \nabla
  w\cdot
  \nabla p_\kappa)-\frac{\kappa}{r^{n-1+2\kappa}}\int_{\partial B_r}
  (w^2+2wp_\kappa)\\
  &=\frac{1}{r^{n-2+2\kappa}}\int_{B_r} |\nabla
  w|^2-\frac{\kappa}{r^{n-1+2\kappa}}\int_{\partial B_r}
  w^2+\frac{2}{r^{n-1+2\kappa}}\int_{\partial B_r}w(x\cdot\nabla
  p_\kappa-\kappa p_\kappa)\\
  &=\frac{1}{r^{n-2+2\kappa}}\int_{B_r} |\nabla
  w|^2-\frac{\kappa}{r^{n-1+2\kappa}}\int_{\partial B_r}
  w^2\\
  &= \frac{1}{r^{n-2+2\kappa}}\int_{B_r}{(-w\Delta
    w)}+\frac{1}{r^{n-1+2\kappa}}\int_{\partial B_r} w(x\cdot\nabla
  w-\kappa w).
\end{align*}
On the other hand we have
\begin{align*}
  \frac{d}{dr}\left(\frac{1}{r^{n-1+2\kappa}}\int_{\partial B_r}
    w^2(x) \right)=&\frac{d}{dr}\int_{\partial B_1}
  \frac{w^2(ry)}{r^{2\kappa}}\\
  &=\int_{\partial B_1}\frac{2 w(ry)(ry\cdot\nabla
    w(ry)-\kappa w(ry))}{r^{2\kappa+1}}\\
  &=\frac{2}{r^{n+2\kappa}}\int_{\partial B_r} w(x\cdot \nabla w-\kappa w)
\end{align*}
and
$$
w\Delta w=(u-p_\kappa)(\Delta u-\Delta p_\kappa)=-p_\kappa\Delta
u\geq 0
$$
as $\Delta u=0$ off $\{x_n=0\}$ and $\Delta u\leq 0$ and
$p_\kappa\geq 0$ on $\{x_n=0\}$. Combining, we obtain
$$
\frac{d}{dr} M_\kappa(r,u,p_\kappa)\geq \frac{2}{r}
W_\kappa(r,u)\geq 0.
$$
\end{proof}

\section{Singular set: proofs}
\label{sec:singular-points}

We now apply the monotonicity formulas in the previous sections to
study the singular points; in particular, we give the proofs of
Theorems~\ref{thm:k-diff-sing-p} and \ref{thm:sing-points}. We start
with establishing the correct growth rate at such points.

\begin{lemma}[Growth estimate]\label{lem:gr-est-above} Let $u\in\S$ and
  $0\in\Gamma_\kappa(u)$. There exists $C>0$ such that
  $$
  |u(x)|\leq C |x|^\kappa\quad\text{in }B_1.
  $$
\end{lemma}
\begin{proof} This is already proved in Caffarelli, Salsa, Silvestre
  \cite{CSS}*{Lemma 6.6}  but for
  the reader's convenience we provide a proof. From \eqref{eq:H'}, \eqref{eq:D-int-parts} we have
  $$
  \frac{H'(r)}{H(r)}=\frac{n-1+2N(r)}{r}\geq\frac{n-1+2\kappa}{r}.
  $$
  Hence
  $$
  \log\frac{H(1)}{H(r)}\geq (n-1+2\kappa)\log\frac1{r}
  $$
  which implies
  $$
  H(r)\leq H(1) r^{n-1+2\kappa}.
  $$
  Finally, the $L^\infty$ bound follows from the fact that $u^+$
  and $u^-$ are actually subharmonic, see e.g.\ \cite{AC}*{Lemma 1}.
\end{proof}

\begin{lemma}[Nondegeneracy at singular points]\label{lem:nondeg-sing} Let $u\in\S$ and
  $0\in\Sigma_\kappa (u)$. There exists $c>0$, possibly depending on $u$, such that
  $$
  \sup_{\partial B_r} |u(x)|\geq c\,r^\kappa,\quad\text{for } 0<r<1.
  $$
\end{lemma}
\begin{proof} Assume the contrary. Then for a sequence $r=r_j\to 0$
  one has
  $$
  h_r:=\left(\frac{1}{r^{n-1}}\int_{\partial B_r}
    u^2\right)^{1/2}=o(r^{\kappa}).
  $$
  Passing to a subsequence if necessary we may assume that
  $$
  u_r(x)=\frac{u(rx)}{h_r}\to q_\kappa(x)\quad\text{uniformly on
  }\partial B_1
  $$
  for some $q_\kappa\in\P_\kappa$, see
  Theorem~\ref{thm:blowup-sing}. Note that $q_\kappa$ is nonzero
  as it must satisfy $\int_{\partial B_1}q_\kappa^2=1$. Now consider
  the functional $M_\kappa(r,u,q_\kappa)$ with $q_\kappa$ as
  above. From the assumption on the growth of $u$ it is easy to
  realize that
  $$
  M_\kappa(0+,u, q_\kappa)=\int_{\partial B_1}
  q_\kappa^2=\frac{1}{r^{n-1+2\kappa}} \int_{\partial B_r} q_\kappa^2
  $$
  Hence, we have that
  $$
  \frac{1}{r^{n-1+2\kappa}} \int_{\partial B_r} (u-q_\kappa)^2\geq
  \frac{1}{r^{n-1+2\kappa}} \int_{\partial B_r} q_\kappa^2
  $$
  or equivalently
  $$
  \int_{\partial B_r} u^2-2 u q_\kappa\geq 0.
  $$
  On the other hand, rescaling, we obtain
  $$
  \int_{\partial B_1} h_r^2 u_r^2-2h_r r^{\kappa} u_r q_\kappa\geq
  0.
  $$
  Factoring out $h_r r^\kappa$, we have
  $$
  \int_{\partial B_1} \frac{h_r}{r^\kappa} u_r^2-2u_r q_\kappa\geq
  0,
  $$
  and passing to the limit over $r=r_j\to 0$
  $$
  -\int_{\partial B_1} q_\kappa^2\geq 0.
  $$
 Since $q_\kappa\not=0$, we have thus reached a contradiction.
\end{proof}

\begin{lemma}[$\Sigma_k(u)$ is $F_\sigma$]\label{lem:sing-set-F-sigma} For any $u\in \S$, the set $\Sigma_\kappa(u)$ is of type $F_\sigma$, i.e., it is a union of countably many closed sets.
\end{lemma}
\begin{proof} Let $E_{j}$ be the set of points  $x_0\in\Sigma_\kappa(u)\cap \overline{B_{1-1/j}}$ such that
\begin{equation}\label{eq:Ej}
\tfrac1j\, \rho^{\kappa}\leq \sup_{|x-x_0|=\rho} |u(x)| < j \rho^\kappa
\end{equation}
for $0<\rho<1-|x_0|$. Note that from Lemmas~\ref{lem:gr-est-above} and \ref{lem:nondeg-sing} we have that 
$$
\Sigma_\kappa(u)=\bigcup_{j=1}^{\infty} E_j.
$$
The lemma will follow once we show that $E_j$ is a closed set. Indeed, if $x_0\in \overline E_j$ then $x_0$ satisfies \eqref{eq:Ej} and we need to show only that $x_0\in \Sigma_\kappa(u)$, or equivalently that $N^{ x_0}(0+,u)=\kappa$ by Theorem~\ref{thm:blowup-sing}. Since the function $x\mapsto N^{x}(0+,u)$ is upper semicontinous, we readily have that  $N^{x_0}(0+,u)\geq \kappa$. On the other hand, if $N^{ x_0}(0+,u)=\kappa'>\kappa$, we would have 
$$
|u(x)|< C|x-x_0|^{\kappa'}\quad\text{in }B_{1-|x_0|}(x_0),
$$
which would contradict the estimate from below in \eqref{eq:Ej}. Thus $N^{x_0}(0+,u)=\kappa$ and therefore $x_0\in E_j$.
\end{proof}

\begin{theorem}[Uniqueness of the homogeneous blowup at singular
  points]\label{thm:uniq-blowup-sing} Let $u\in\S$ and
  $0\in\Sigma_\kappa(u)$. Then there exists a unique nonzero
  $p_\kappa\in\P_\kappa$ such that
  $$
  u_r^{(\kappa)}(x):=\frac{u(rx)}{r^\kappa}\to p_\kappa(x).
  $$
\end{theorem}

\begin{proof} Let $u^{(\kappa)}_r(x)\to u_0(x)$ in
  $C^{1,\alpha}_{\loc}(\R^n)$ over a certain subsequence $r=r_j\to
  0+$.  The
  existence of such limit follows from the
  growth estimate $|u(x)|\leq C|x|^\kappa$.  We call such $u_0$ a \emph{homogeneous
  blowup}, in contrast to the blowups that were based on the scaling
  \eqref{eq:rescaling}.  Note that Lemma~\ref{lem:nondeg-sing} implies that
  $u_0$ in not identically zero. Next, we have for any $r>0$
  $$
  W_\kappa(r, u_0)=\lim_{r_j\to 0+} W_\kappa(r,
  u_{r_j}^{(\kappa)})=\lim_{r_j\to 0+} W_\kappa(rr_j,
  u)=W_\kappa(0+,u)=0.
  $$
  In view of Theorem~\ref{thm:weiss} this implies that the harmonic function $u_0$ is homogeneous of degree
  $\kappa$. Repeating the arguments in
  Theorem~\ref{thm:blowup-sing}, we see that $u_0$ must be a
  polynomial in $\P_\kappa$. (The same could be achieved by looking at $N(r,
  u_0)$).

 We now apply Monneau's monotonicity formula to the pair $u$,
  $u_0$. By Theorem~\ref{thm:Monn-mon-form} the limit
  $M_\kappa(0+,u,u_0)$ exists and can be computed by
  $$
  M_\kappa(0+, u, u_0)=\lim_{r_j\to 0+} M_\kappa(r_j, u,
  u_0)=\lim_{j\to\infty} \int_{\partial B_1} (u^{(\kappa)}_{r_j}-u_0)^2=0.
  $$
  In particular, we obtain that
  $$
  \int_{\partial B_1} (u^{(\kappa)}_r(x)-u_0)^2= M_\kappa(r,
  u,u_0)\to 0
  $$
  as $r\to 0+$ (not just over $r=r_j\to 0+$!). Thus, if $u_0'$ is a
  limit of $u_r^{(\kappa)}$ over another sequence $r=r_j'\to
  0$, we obtain that
  $$
  \int_{\partial B_1} (u_0'-u_0)^2=0.
  $$
  Since both $u_0$ and $u_0'$ are homogeneous of degree $\kappa$,
  they must coincide in $\R^n$.
\end{proof}

We note explicitly that the conclusion of Theorem
\ref{thm:uniq-blowup-sing} is equivalent to the Taylor expansion
$$
u(x)=p_\kappa(x)+o(|x|^\kappa)
$$
and therefore it proves the first part of
Theorem~\ref{thm:k-diff-sing-p}. The next result is essentially the
second part of Theorem~\ref{thm:k-diff-sing-p}.

\begin{theorem}[Continuous dependence of the
  blowups]\label{thm:cont-dep-blowup} Let $u\in\S$. For
  $x_0\in\Sigma_\kappa(u)$ denote by $p^{x_0}_\kappa$ the blowup of
  $u$ at $x_0$ as in Theorem~\ref{thm:uniq-blowup-sing}, so that
  $$
  u(x)=p^{x_0}_\kappa(x-x_0)+o(|x-x_0|^\kappa).
  $$
  Then the mapping $x_0\mapsto p_\kappa^{x_0}$ from
  $\Sigma_\kappa(u)$ to $\P_\kappa$ is continuous. Moreover, for any compact $K\subset \Sigma_\kappa(u)\cap B_{1}$ there
  exists a modulus of continuity $\sigma_K$, $\sigma_K(0+)=0$ such that
  $$
  |u(x)-p^{x_0}_\kappa(x-x_0)|\leq \sigma_K (|x-x_0|)|x-x_0|^\kappa
  $$
  for any $x_0\in K$.
\end{theorem}

\begin{proof} Note that since $\P_\kappa$ is a convex subset of a finite-dimensional
vector space, namely the space of all $\kappa$-homogeneous
polynomials, all the norms on such space are equivalent. We can then
endow $\P_\kappa$ with the norm of $L^2(\partial B_1)$.

This being said, the proof is similar to that of the last part of
the previous theorem. Given $x_0\in \Sigma_\kappa(u)$ and $\epsilon>0$ fix $r_\epsilon=r_\epsilon(x_0)$ such that
$$
M_\kappa^{x_0}(r_\epsilon, u,
p_\kappa^{x_0}):=\frac{1}{r_\epsilon^{n-1+2\kappa}}\int_{\partial
  B_{r_\epsilon}} (u(x+x_0)-p^{x_0}_\kappa)^2<\epsilon.
$$
There exists $\delta_\epsilon=\delta_\epsilon(x_0)$ such that if
$x_0'\in\Sigma_\kappa(u)$ and $|x_0'-x_0|<\delta_\epsilon$, then
$$
M_\kappa^{x_0'}(r_\epsilon, u,
p_\kappa^{x_0})=\frac{1}{r_\epsilon^{n-1+2\kappa}}\int_{\partial
  B_{r_\epsilon}} (u(x+x_0')-p^{x_0}_\kappa)^2<2\epsilon.
$$
From the monotonicity of the Monneau's functional, we will have
that
$$
M_\kappa^{x_0'}(r, u, p_\kappa^{x_0})< 2\epsilon,\quad
0<r<r_\epsilon.
$$
Letting $r\to 0$, we will therefore obtain
$$
M^{x_0'}(0+,u,p_\kappa^{x_0})=\int_{\partial B_1}
(p_\kappa^{x_0'}-p_\kappa^{x_0})^2\leq 2\epsilon.
  $$
This shows the first part of the theorem.

To show the second part, we notice that we have
\begin{align*}
  \|u(\cdot+x_0')-p^{x_0'}_\kappa\|_{L^2(\partial B_r)}&\leq
  \|u(\cdot+x_0')-p^{x_0}_\kappa\|_{L^2(\partial B_r)}
  +\|p^{x_0}_\kappa-p^{x_0'}_\kappa\|_{L^2(\partial B_r)}\\
  &\leq
  2(2\epsilon)^{\frac12}r^{\frac{n-1}2+\kappa},
\end{align*}
for $|x_0'-x_0|<\delta_\epsilon$, $0<r<r_\epsilon$, or equivalently
\begin{equation}\label{eq:w-p-L2}	
\| w^{x_0'}_r-p^{x_0'}_\kappa\|_{L^2(\partial B_1)}\leq 2(2\epsilon)^{\frac12},
\end{equation}
where
$$
w^{x_0'}_r(x):=\frac{u(rx+x_0')}{r^\kappa}.
$$
Now, covering the compact $K\subset \Sigma_\kappa(u)\cap B_1$ with finitely many balls $B_{\delta_\epsilon(x_0^i)}(x_0^i)$ for some $x_0^i\in K$, $i=1,\ldots,N$, we obtain that \eqref{eq:w-p-L2} is satisfied for all $x_0'\in K$ with $r<r_\epsilon^K:=\min\{r_\epsilon(x_0^i) \mid i=1,\ldots, N\}$.
Now notice that
$$
w^{x_0'}_r\in\S,\quad p^{x_0'}_\kappa\in\S
$$
with uniformly bounded $C^{1,\alpha}(\overline{B_1})$ norms. The
solutions of the Signorini problem
\eqref{eq:signorini-1}--\eqref{eq:0-fbp} enjoy the uniqueness property
in the sense that they coincide if they have the same trace on
$\partial B_1$. Thus, arguing by contradiction and using a compactness
argument, we can establish the estimate
$$
\|w^{x_0'}_r-p^{x_0'}_\kappa(x)\|_{L^\infty(B_{1/2})}\leq
C_{\epsilon}
$$
for $x_0'\in K$, $0<r<r_\epsilon^K$ with
$C_{\epsilon}\to 0$ as $\epsilon\to 0$. Clearly, this implies the second part of the theorem. 
\end{proof}

We are now ready to prove Theorem~\ref{thm:sing-points} on the
structure of the singular set.

\begin{proof}[Proof of Theorem~\ref{thm:sing-points}] First, recall that the equality $\Sigma_\kappa(u)=\Gamma_\kappa(u)$ for $\kappa=2m$, $m\in\N$, is proved in Theorem~\ref{thm:blowup-sing}.

The proof of the structure of $\Sigma_\kappa^d(u)$ is based on two classical results in analysis: Whitney's
extension theorem \cite{Wh} and the implicit function theorem. This
parallels the approach in \cite{Ca2} for the classical obstacle
problem.

  \medskip \emph{Step 1: Whitney's extension.} 
Let $K=E_j$ be the compact subset of $\Sigma_\kappa(u)$ defined in the proof of Lemma~\ref{lem:sing-set-F-sigma}.
Write the polynomials
  $p^{x_0}_\kappa$ in the expanded form
  $$
  p^{x_0}_\kappa(x)=\sum_{|\alpha|=\kappa}
  \frac{a_\alpha(x_0)}{\alpha!}x^\alpha.
  $$
  Then the coefficients $a_\alpha(x)$ are continuous on
  $\Sigma_{\kappa}(u)$ by Theorem~\ref{thm:cont-dep-blowup}. Moreover, since $u(x)=0$ on
  $\Sigma_\kappa(u)$, we have
  $$
  |p^{x_0}_\kappa(x-x_0)|\leq \sigma(|x-x_0|)|x-x_0|^\kappa,\quad\text{for }
  x, x_0\in K.
  $$
  For any multi-index $\alpha$, $|\alpha|\leq \kappa$, define
  $$
  f_\alpha(x)=\begin{cases} 0 & |\alpha|<\kappa\\ a_\alpha(x) &
    |\alpha|=\kappa,
\end{cases}\qquad x\in\Sigma_k(u).
$$
We claim that the following compatibility condition is satisfied, which will enable us to apply the Whitney's extension theorem.

\begin{lemma}\label{lem:Whitney-compat}
For any $x_0,x\in K$
\begin{equation}\label{eq:compat-1}
f_\alpha(x)=\sum_{|\beta|\leq
  \kappa-|\alpha|}\frac{f_{\alpha+\beta}(x_0)}{\beta!}(x-x_0)^\beta +R_\alpha(x,x_0)
\end{equation}
with
\begin{equation}\label{eq:compat-2}
|R_\alpha(x,x_0)|\leq \sigma_\alpha(|x-x_0|)|x-x_0|^{\kappa-|\alpha|},
\end{equation}
where $\sigma_\alpha=\sigma_\alpha^K$ is a certain modulus of continuity.
\end{lemma}

\begin{proof}
1) Consider first the case $|\alpha|=\kappa$. Then we have
$$
R_\alpha(x,x_0)=a_\alpha(x)-a_\alpha(x_0)
$$
and therefore $|R_\alpha(x,x_0)|\leq \sigma_\alpha(|x-x_0|)$ from continuity of the mapping $x\mapsto p^x$ on $K$.

\smallskip
2) For $0\leq |\alpha|<\kappa$ we have

$$
R_\alpha(x,x_0)=-\sum_{\gamma>\alpha, |\gamma|=\kappa} \frac{a_\gamma(x_0)}{(\gamma-\alpha)!}(x-x_0)^{\gamma-\alpha}= - \partial^\alpha p^{x_0}_\kappa (x-x_0).
$$
Now suppose that there exists no modulus of continuity $\sigma_\alpha$ such that \eqref{eq:compat-2} is satisfied for all $x_0, x\in K$. Then there exists $\delta>0$ and a sequence $x_0^i, x^i\in K$ with 
$$
|x^i-x_0^i|=:\rho_i\to 0
$$
such that
\begin{equation}\label{eq:compat-contr}
\Big|\sum_{\gamma>\alpha, |\gamma|=\kappa} \frac{a_\gamma(x_0^i)}{(\gamma-\alpha)!}(x^i-x_0^i)^{\gamma-\alpha}\Big|\geq \delta |x^i-x_0^i|^{\kappa-|\alpha|}.
\end{equation}
Consider the rescalings
$$
w^i(x)=\frac{u(x_0^i+\rho_i x)}{\rho_i^\kappa},\quad \xi^i=(x^i-x_0^i)/\rho_i.
$$
Without loss of generality we may assume that $x_0^i\to x_0\in K$ and $\xi^i\to \xi_0\in \partial B_1$.
From Theorem~\ref{thm:cont-dep-blowup} we have that
$$
|w^i(x)-p^{x_0^i}_\kappa(x)|\leq \sigma(\rho_i|x|)|x|^\kappa
$$
and therefore $w^i(x)$ converges locally uniformly in $\R^n$ to $p^{x_0}_\kappa(x)$.
Further, note that since $x^i$ and $x_0^i$ are from the set $K=E_j$, the inequalities \eqref{eq:Ej} are satisfied there. Moreover, we also have that similar inequalities are satisfied for the rescaled function $w^i$ at $0$ and $\xi^i$. Therefore, passing to the limit, we obtain that
$$
\tfrac1j \rho^\kappa\leq \sup_{|x-\xi_0|=\rho} p^{x_0}_\kappa(x)\leq j \rho^\kappa,\quad 0<\rho<\infty.
$$
This implies that $\xi_0$ is a point of frequency $\kappa=2m$ for the polynomial $p^{x_0}_\kappa$ and by Theorem~\ref{thm:blowup-sing} we have that $\xi_0\in \Sigma_\kappa (p^{x_0}_\kappa)$. In particular,
$$
\partial^\alpha p^{x_0}_\kappa(\xi_0)=0,\quad\text{for }|\alpha|<\kappa.
$$
However, dividing both parts of \eqref{eq:compat-contr} by $\rho_i^{\kappa-|\alpha|}$ and passing to the limit, we obtain that
$$
|\partial^\alpha p^{x_0}_\kappa(\xi_0)|=\Big|\sum_{\gamma>\alpha, |\gamma|=\kappa} \frac{a_\gamma(x_0)}{(\gamma-\alpha)!}(\xi_0)^{\gamma-\alpha}\Big|\geq \delta,
$$
a contradiction.
\end{proof}

So in all cases, the compatibility conditions \eqref{eq:compat-2} are satisfied and we
can apply Whitney's extension theorem. Thus, there exists a function
$F\in C^\kappa(\R^n)$ such that
$$
\partial^\alpha F=f_\alpha\quad\text{on } E_j
$$
for any $|\alpha|\leq \kappa$.

\medskip \emph{Step 2: Implicit function theorem.} Suppose now $x_0\in
\Sigma_\kappa^d(u)\cap E_j$. Recalling Definition~\ref{def:dim-sing-point} this means that
$$
d=\dim\{\xi\in\R^{n-1} \mid \xi\cdot \nabla_{x'} p^{x_0}_\kappa\equiv
0\}.
$$
Then there are $n-1-d$ linearly independent unit vectors
$\nu_i\in\R^{n-1}$, $i=1$, \ldots, $n-1-d$, such that
$$
\nu_i\cdot\nabla_{x'}p^{x_0}_\kappa\not=0\quad\text{on }\R^{n}.
$$
This implies that there exist multi-indices $\beta^i$ of order
$|\beta^i|=\kappa-1$ such that
$$
\partial_{\nu_i}(\partial^{\beta^i}p^{x_0}_\kappa)(0)\not=0.
$$
This can be written as
\begin{equation}\label{eq:cond-impl-f-thm}
\partial_{\nu_i}\partial^{\beta^i}F(x_0)\not=0,\quad i=1,\ldots, n-1-d.
\end{equation}
On the other hand,
$$
\Sigma^d_\kappa(u)\cap E_j \subset \bigcap_{i=1}^{n-1-d}
\{\partial^{\beta^i} F=0\}.
$$
Therefore, in view of the implicit function theorem, the condition
\eqref{eq:cond-impl-f-thm} implies  that $\Sigma^d_\kappa(u)\cap E_j$ is
contained in a $d$-dimensional manifold in a neighborhood of $x_0$. Finally, since $\Sigma_k(u)=\bigcup_{j=1}^{\infty} E_j$ this implies the statement of the theorem.
\end{proof}

\part{Nonzero Thin Obstacle}
\label{part:nonz-thin-obst}

\section{Normalization}
\label{sec:normalization-1}

We now want to study the Signorini problem with a not necessarily
zero thin obstacle $\phi$ defined on a flat portion of the boundary.
More precisely, given a function $\phi\in C^{k,1}(B_1')$, for some
$k\in \mathbb N$, we consider the unique minimizer in the Signorini
problem in $B_1^+$, with thin obstacle $\phi$. Such $v$ satisfies
\begin{gather}
\label{eq:signorini-1-nonzero}
\Delta v=0\quad\text{in } B_1^+\\
\label{eq:signorini-2-nonzero}
v-\phi\geq 0,\quad -\partial_{x_n} v\geq 0,\quad
(v-\phi)\,\partial_{x_n} v=0\quad
\text{on }B_1'\\
\label{eq:0-fbp-nonzero}
0\in\Gamma(v):=\partial\{v(\cdot,0)-\phi=0\},
\end{gather}
where \eqref{eq:signorini-1-nonzero} is to be interpreted in the
weak sense.
\begin{definition}
We say that $v\in C^{1,\alpha}(B_1^+\cup B_1')$ belongs to the class
$\S^\phi$ if it satisfies
\eqref{eq:signorini-1-nonzero}--\eqref{eq:0-fbp-nonzero}.
\end{definition}

The basic idea now is considering the difference
$u(x',x_n)=v(x',x_n)-\phi(x')$. The complication is that $u$ is no
longer harmonic in $B_1^+$, but instead satisfies
$$
\Delta u=-\Delta_{x'}\phi.
$$
This introduces a certain error in the computations that potentially
could prevent us from successfully studying nonregular points. Thus
we need a slightly refined argument that will enable us to control
the error.

\subsection{Subtracting the Taylor polynomial}
\label{sec:subtr-tayl-polyn}

\begin{lemma}[Harmonic extension of homogeneous polynomials]\label{lem:polynom-ext}
  Let $q_k(x')$ be a homogeneous polynomial of degree $k$ on
  $\R^{n-1}$. There exists a unique homogeneous polynomial
  $\tilde q_k$ of degree $k$ on $\R^{n}$ such that
\begin{align*}
  \Delta \tilde q_k=0&\quad\text{in }\R^{n},\\
  \tilde q_k(x',0)=q_k(x'),&\quad\text{for any }x'\in\R^{n-1},\\
  \tilde q_k(x',-x_n)=\tilde q_k(x',x_n)&\quad\text{for any
  }x'\in\R^{n-1},\ x\in\R.
\end{align*}

\end{lemma}
\begin{proof}
  1) \emph{Existence}. In the simplest case when $q_k(x')=x_j^k$,\quad
  $j=1,\ldots,n-1$, one can take $\tilde q_k(x)=\Re(x_j+i\,x_n)^k$.
Arguing in analogy with this situation, for
\begin{equation}\label{eq:simple-homogen-poly}
  q_k(x')=(e'\cdot x')^k,\quad\text{where } e'\in\R^{n-1},\ |e'|=1
\end{equation}
one can take
$$
\tilde q_k(x)=\Re(e'\cdot x'+i\,x_n)^k.
$$
Now the existence for an arbitrary polynomial of order $k$ follows
from the fact that any homogeneous polynomial of degree $k$ is a
linear combination of those of the form
\eqref{eq:simple-homogen-poly}.

2) \emph{Uniqueness}. By the linearity of the Laplacian, it is
sufficient to show that the only extension of $q_k=0$ is $\tilde
q_k=0$. Note that for any such extension both $\tilde q_k$ and
$\partial_{x_n}\tilde q_\kappa$ will vanish on $\R^{n-1}\times\{0\}$
(the latter following from even symmetry in $x_n$). Since $\tilde
q_k$ is also harmonic, by the Cauchy-Kovalevskaya theorem it must
vanish identically.
\end{proof}

Assume now that the lower dimensional obstacle is given by $\phi\in
C^{k,1}(B_1')$. Let $Q_k(x')$ be the Taylor polynomial of degree $k$
of $\phi$ at the origin , i.e.,
$$
\phi(x')=Q_k(x')+O(|x'|^{k+1}).
$$
Moreover, we will also have
$$
\Delta_{x'} \phi(x')=\Delta_{x'}Q_k(x')+O(|x'|^{k-1}).
$$
Representing $Q_k=\sum_{m=0}^k q_m$, where $q_m$ are homogeneous
polynomials of degree $m$, by Lemma~\ref{lem:polynom-ext} we can find
a harmonic extension $\tilde Q_k$ of $Q_k$ into $\R^n$. For the
solution $v$ of
\eqref{eq:signorini-1-nonzero}--\eqref{eq:0-fbp-nonzero} consider the
difference
$$
u(x',x_n):=v(x)-\tilde Q_k(x',x_n)-(\phi(x')-Q_k(x')).
$$
It is easy to see that $u$ satisfies
\begin{gather}
\label{eq:signorini-1-diff}
|\Delta u|=|\Delta_{x'}(\phi-Q_k)|\leq M|x'|^{k-1} \quad\text{in } B_1^+\\
\label{eq:signorini-2-diff}
u\geq 0,\quad -\partial_{x_n} u\geq 0,\quad u\,\partial_{x_n} u=0\quad
\text{on }B_1'\\
\label{eq:0-fbp-diff}
0\in\Gamma(u):=\partial\{u(\cdot,0)=0\}.
\end{gather}

\begin{definition} We say that $u\in C^{1,\alpha}(B_1^+\cup B_1')$
  belongs to the class $\S_k(M)$ if it satisfies
  \eqref{eq:signorini-1-diff}--\eqref{eq:0-fbp-diff} and moreover
  $$
  \|u\|_{C^1(B_1)}\leq M.
  $$
  We will use the full notation $\S_k(M)$ if the value of the
  constant $M$ is important. In all other cases we will denote this
  class simply by $\S_k$.
\end{definition}

As before, we may assume that $u\in\S_k$ is automatically extended to
$B_1$ by even symmetry
$$
u(x',-x_n)=u(x',x_n).
$$
For this extension, the distributional Laplacian $\Delta u$ is a sum
of a nonpositive measure supported in $B'_1$ and an $L^\infty$
function in $B_1$. More precisely, integrating by parts in $B_1^\pm$
and using \eqref{eq:signorini-1-diff}--\eqref{eq:signorini-2-diff}, we
have
$$
|\Delta u|\leq M|x'|^{k-1}+2|\partial_{x_n} u|
\H^{n-1}\big|_{B_1'}\quad\text{in }\D'(B_1).
$$

\section{Generalized frequency formula}
\label{sec:gener-freq-form}

By allowing nonzero obstacles one sacrifices Almgren's frequency
formula in its purest form. However, the following modified version
holds. In the case $k=2$ Theorem \ref{thm:gen-almgren} below has
first been established by Caffarelli, Salsa, Silvestre \cite{CSS}. For
their purposes they only needed to consider the class $\S_2$ since
it allows to capture the slowest growth rate of the solution at a
regular free boundary point and thus establish the optimal
regularity. For singular free boundary points, instead, we need to
consider the full range of values of $k$.

\begin{theorem}[Generalized Frequency Formula]\label{thm:gen-almgren} Let $u\in \S_k(M)$. With $H(r)$ as in
  \eqref{eq:IH} there exist $r_M>0$ and $C_M>0$ such that
  $$
 r \mapsto \Phi_k(r,u):=(r+C_M r^2)\frac{d}{dr}\log\max\left\{H(r),
    r^{n-1+2k}\right\},
  $$
 is nondecreasing for $0<r<r_M$.
\end{theorem}

The rest of this section is devoted to proving
Theorem~\ref{thm:gen-almgren} and is rather technical. The reader
might want to skip it, at least in the first reading, and proceed
directly to the next section.

\subsection{Proof of Theorem \ref{thm:gen-almgren}}
In order to prove Theorem \ref{thm:gen-almgren} we first establish
two auxiliary lemmas. For $u\in \S_k(M)$, with $D(r)$ and $H(r)$ as
in \eqref{eq:IH}, we also consider the following quantities
\begin{equation}\label{eq:GI}
  G(r) := \int_{B_r} u^2 ,\quad
  I(r) := \int_{\partial B_r} u u_\nu=\int_{B_r} |\nabla u|^2
  +\int_{B_r} u\Delta u.
\end{equation}

\begin{lemma}\label{lem:deriv} For $u\in\S_k$ we have the following identities
\begin{align*}
  G'(r)&=H(r),\\
  H'(r)&=\frac{n-1}{r}\,H(r)+2\int_{\partial B_r} u u_\nu,\\
  D'(r)&=\frac{n-2}{r}\,D(r)+2\int_{\partial B_r}
  u_\nu^2-\frac2r\int_{B_r} (x\cdot\nabla u)\Delta u,\\
  I'(r) &=\frac{n-2}{r}\,I(r)+2\int_{\partial B_r}
  u_\nu^2-\frac2r\int_{B_r}(x\cdot\nabla u)\Delta u\\
  &\phantom{\mbox{}=\mbox{}}-\frac{n-2}r\int_{B_r} u\Delta
  u+\int_{\partial B_r} u\Delta u.
\end{align*}
\end{lemma}
\begin{proof} The proof of these statements is rather standard and is
  therefore omitted. The reader may consult \cite{GL}, \cite{GL2} for similar
  computations.
\end{proof}

\begin{remark}\label{rem:discard} Note that both $u$ and $x\cdot \nabla u$ vanish
  continuously on $\supp\partial_{x_n} u\big|_{B_1'}=\supp \Delta
  u\big|_{B_1'}$, hence we can discard the integrals of $(x \cdot\nabla u)\Delta u$
  and $u\Delta u$ over $B_r'$ and $\partial B_r'$.
\end{remark}

\begin{lemma}\label{lem:H-D-est}
  For any $u\in\S_k(M)$ we have the following estimates
\begin{align}\label{eq:est-3}
  \int_{\partial B_r}u^2&\leq C_M r\int_{B_r} |\nabla u|^2+ C_M
  r^{n+2k+1}\\
\label{eq:est-4}
\int_{B_r}u^2&\leq C_M r^2\int_{B_r} |\nabla u|^2+ C_M r^{n+2k+2}
\end{align}
\end{lemma}
Equivalently, we may write these inequalities as
\begin{align}\label{eq:est-3'}
  H(r)&\leq C_M r D(r)+ C_Mr^{n+2k+1}\\
\label{eq:est-4'}
G(r)&\leq C_M r^2 D(r)+ C_Mr^{n+2k+2}
\end{align}

\begin{proof} These inequalities are essentially established in
  \cite{CSS}. Below we outline the main steps with references to the
  corresponding lemmas in \cite{CSS}.

  One starts with the well-known trace inequality
\begin{equation}\label{eq:est-1}
\int_{\partial B_r} |u(x)-\bar u_r|^2\leq C r\int_{B_r} |\nabla
u|^2,
\end{equation}
where
$$
\bar u_r=\dashint_{\partial
  B_r}u.
$$
Next, since $u\in\S_k$, by \cite{CSS}*{Lemma 2.9}
one has
\begin{equation}\label{eq:est-2}
u(0)\geq \dashint_{\partial B_r} u -C r^{k+1}.
\end{equation}
Now, \eqref{eq:est-1} gives
$$
\int_{\partial B_r} u^2\leq C r \int_{B_r} |\nabla u|^2+ 2\bar u_r
\int_{\partial B_r} u.
$$
Further, \eqref{eq:est-2} gives (since $u(0)=0$)
$$
\dashint_{\partial B_r} u\leq C r^{k+1},
$$
which implies
$$
\int_{\partial B_r} u^+\leq \int_{\partial B_r} u^- + C r^{n+k}.
$$
On the other hand,
$$
\int_{\partial B_r} u^-\leq C r^{\frac{n}2}\left(\int_{B_r}|\nabla
  u|^2\right)^{\frac12},
$$
see \cite{CSS}*{Lemma 2.13}. Hence
$$
\int_{\partial B_r} |u|\leq C r^{\frac{n}2}\left(\int_{B_r}|\nabla
  u|^2\right)^{\frac12}+ C r^{n+k}.
$$
Therefore
\begin{align*}
  \int_{\partial B_r} u^2&\leq C r\int_{B_r} |\nabla
  u|^2+\frac{C}{r^{n-1}}\left(\int_{\partial B_r} |u|\right)^2\\
  &\leq C r\int_{B_r} |\nabla u|^2+ C r^{n+2k+1}.
\end{align*}
This proves \eqref{eq:est-3}. Integrating in $r$, we obtain
\eqref{eq:est-4}.
\end{proof}

We are now ready to prove the Generalized Frequency Formula.

\begin{proof}[Proof of Theorem \ref{thm:gen-almgren}] 1) First we want to make a remark on the definition of
  $\Phi_k(r,u)$. The functions $H(r)$ and $r^{n-1+2k}$ are
  continuously differentiable and therefore the function $\max\{H(r),
  r^{n-1+2k}\}$ is absolutely continuous or, equivalently, belongs to
  the
  Sobolev space $W^{1,1}_\loc((0,1))$. It follows that $\Phi_k$
  is uniquely identified only up to a set of measure zero. The
  monotonicity of $\Phi_k$ should be understood in the sense that
  there exists a monotone increasing function which equals
  $\Phi_k$ almost everywhere. Therefore, without loss of generality we
  may assume that
  $$
  \Phi_k(r,u)=(r+C_Mr^2)\frac{d}{dr}\log
  r^{n-1+2k}=(n-1+2k)(1+C_Mr)
  $$
  on $F:=\{r\in(0,1) \mid H(r)\leq r^{n-1+2k}\}$ and
  $$
  \Phi_k(r,u)=(r+C_Mr^2)\frac{d}{dr}\log
  H(r)=(r+C_Mr^2)\frac{H'(r)}{H(r)}
  $$
  on $U:=\{r\in(0,1) \mid H(r)>r^{n-1+2k}\}$. Following an idea introduced in \cite{GL} we now note that it will
  be enough to check that $\Phi_k'(r,u)>0$ in $U$. Indeed, it is clear
  that $\Phi_k$ is monotone on $F$ and if $(r_0,r_1)$ is a maximal
  open interval in $U$, then $r_0,r_1\in F$ and we will have that
  $$
  \Phi_k(r_0,u)\leq \Phi_k(r_0+,u)\leq \Phi_k(r_1-,u)\leq
  \Phi_k(r_1,u).
  $$
  Therefore, we will concentrate only on the set $U$.

  \medskip 2) Now suppose $r\in(0,1)$ is such that $H(r)>
  r^{n-1+2k}$. Using \eqref{eq:GI} and the second identity in
  Lemma~\ref{lem:deriv} we find
\begin{align*}
  \Phi_k(r,u)&=(r+C_M r^2)\frac{H'(r)}{H(r)}\\
  &=(r+C_M r^2)\left(\frac{n-1}r + 2\frac{I(r)}{H(r)}\right)\\
  &=(n-1)(1+C_M r)+ 2r(1+C_M r)\frac{I(r)}{H(r)}.
\end{align*}
Since $(n-1)(1+C_M r)$ is clearly nondecreasing, it will be enough
to show the monotonicity of $r(1+C_M r)\frac{I(r)}{H(r)}$.

From Lemma \ref{lem:deriv} we now have
\begin{align*}
  &\frac{d}{dr}\log\left(r(1+C_M
    r)\frac{I(r)}{H(r)}\right)=\frac1r+\frac{C_M}{1+C_Mr}+\frac{I'(r)}{I(r)}-\frac{H'(r)}{H(r)}\\&=\frac{C_M}{1+C_M
    r}+2\left\{\frac{\int_{\partial B_r}
      u_\nu^2}{\int_{\partial B_r} u u_\nu}-\frac{\int_{\partial
        B_r}
      u u_\nu}{\int_{\partial B_r} u^2}\right\}\\
  &\phantom{\mbox{}=\mbox{}}+\frac{-\frac2r \int_{B_r}(x\cdot \nabla
    u)\Delta u -\frac{n-2}r\int_{B_r} u\Delta u+\int_{\partial B_r}
    u\Delta u}{\int_{\partial B_r} u u_\nu}.
\end{align*}
The expression in curly brackets in the right-hand side of the
latter equation is the same as that in the proof of
Theorem~\ref{thm:almgren} and is nonnegative by the Cauchy-Schwarz
inequality. We thus obtain
\[
\frac{d}{dr}\log\left(r(1+C_M
    r)\frac{I(r)}{H(r)}\right)\geq \frac{C_M}{1+C_Mr} + E(r),
\]
where we have let
$$
E(r):=\frac{\displaystyle -\frac2r \int_{B_r}(x\cdot\nabla
  u)\Delta u  - \frac{n-2}r\int_{B_r} u\Delta u+\int_{\partial B_r}
  u\Delta u}{\displaystyle \int_{\partial B_r} u u_\nu}.
$$
This is the error term that derives from the non-vanishing of
$\Delta u$. Since the first term $C_M/(1+rC_M)$ is greater than a
positive constant for small $r$, to complete the proof of the
theorem it will be enough to show that $E(r)$ is bounded below.

\smallskip
3) \emph{Estimating $E(r)$.} We estimate the denominator and the
numerator of $E$ separately.

\medskip \emph{Denominator}: Using Cauchy-Schwarz and the inequalities
\eqref{eq:est-3'}--\eqref{eq:est-4'}, we have

\begin{align*}
  \int_{\partial B_r} u\partial_\nu u&=\int_{B_r} |\nabla
  u|^2+\int_{B_r}
  u\Delta u\\
  &\geq D(r)-2\left(\int_{B_r^+}
    u^2\right)^{\frac12}\left(\int_{B_r^+}|\Delta u|^2\right)^{\frac12}\\
  &\geq D(r)-C G(r)^{\frac12} r^{\frac{n}2+k-1}\\
  &\geq D(r)-C \left(rD(r)^{\frac12}+
    r^{\frac{n}2+k+1}\right)r^{\frac{n}2+k-1}\\
  &\geq D(r)-C D(r)^{\frac12} r^{\frac{n}2+k}-C r^{n+2k}.
\end{align*}

\medskip \emph{Numerator}: Again using Cauchy-Schwarz and the
inequalities \eqref{eq:est-3'}--\eqref{eq:est-4'}, we have
\begin{align*}
  \left| \frac1r\int_{B_r} u\Delta u\right|&\leq
  \frac2r\left(\int_{B_r^+}
    u^2\right)^{\frac12}\left(\int_{B_r^+}|\Delta
    u|^2\right)^{\frac12}\\
  &\leq CD(r)^{\frac12} r^{\frac{n}2+k-1}+C r^{n+2k-1}
\end{align*}
\begin{align*}
  \left|\frac1r\int_{B_r}\Delta u(x\cdot\nabla u)\right|&\leq
  \frac2r\left(\int_{B_r^+} |\nabla u|^2
    |x|^2\right)^{\frac12}\left(\int_{B_r^+}|\Delta
    u|^2\right)^{\frac12}\\
  &\leq CD(r)^{\frac12} r^{\frac{n}2+k-1}
\end{align*}
\begin{align*}
  \left|\int_{\partial B_r} u\Delta u\right|&\leq
  2\left(\int_{\partial B_r^+}
    u^2\right)^{\frac12}\left(\int_{\partial B_r^+}|\Delta
    u|^2\right)^{\frac12}\\
  &\leq C H(r)^{\frac12}r^{\frac{n-1}2+k-1}\\
  &\leq C\left(r^{\frac12}
    D(r)^{\frac12}+r^{\frac{n}2+k+\frac12}\right)r^{\frac{n-1}2+k-1}\\
  &\leq C D(r)^{\frac12}r^{\frac{n}2+k-1}+C r^{n+2k-1}.
\end{align*}

Now, collecting the estimates on the denominator and the numerator of
$E(r)$ we obtain
$$
|E(r)|\leq
C\frac{D(r)^{\frac12}r^{\frac{n}2+k-1}+r^{n+2k-1}}{D(r)-C
  D(r)^{\frac12} r^{\frac{n}2+k}-C r^{n+2k}}
$$
Finally, recall that we assume
$$
H(r)> r^{n-1+2k}.
$$
Then by Lemma~\ref{lem:H-D-est} we also have
$$
D(r)> c\, r^{n-2+2k}.
$$
The latter inequality now implies that $|E(r)|$ is uniformly bounded
for sufficiently small $r$. This completes the proof of the theorem.
\end{proof}

\section{Growth near the free boundary}\label{sec:growth-near-free}

To get a better sense of the relation between the frequency function
in Theorem~\ref{thm:almgren}, and the generalized frequency in
Theorem~\ref{thm:gen-almgren} above, we note that for $u\in\S_k$, we
have
$$
\Phi_k(r,u)=(1+C_Mr) (n-1+2 N(r,u))\quad \text{if } H(r)>
r^{n-1+2k}.
$$
So one expects $\Phi_k(0+,u)$ to behave similarly to $n-1+2
N(0+,u)$.

\begin{lemma}[Consistency of $\Phi_k(0+,u)$]\label{lem:consist} Let $u\in\S_k$. For any $m$ such that $2\leq m\leq
k$ one has
  $$
  \Phi_m(0+,u)=\min\{\Phi_k(0+,u),n-1+2m\}.
  $$
 If $m=k$ we obtain in particular,
  $$
  \Phi_k(0+,u)\leq n-1+2k.
  $$
\end{lemma}
\begin{proof}
  1) We start with the latter inequality. Let $\kappa$ be such that
  $$
  \Phi_k(0+,u)=n-1+2\kappa.
  $$
  We want to show that $\kappa\leq k$. Observe that,
  in general, if $H(r)<r^{n-1+2k}$ along a sequence $r=r_j\to 0+$, then we must have $\Phi_k(0+,u)=n-1+2k$. Therefore if
  $\kappa\not=k$, we must have for small $r$
  $$
  H(r)\geq r^{n-1+2k},\quad
  \Phi_k(r,u)=(r+C\,r^2)\frac{H'(r)}{H(r)}.
  $$

  Assume now $\kappa>k$. Fix some $\kappa'\in (k,\kappa)$. Then for
  small enough $0<r\leq r_0$
  $$
  r\frac{H'(r)}{H(r)}\geq n-1+2\kappa'.
  $$
  Dividing by $r$ and integrating from $r$ to $r_0$, we obtain
  $$
  \log\frac{H(r_0)}{H(r)}\geq (n-1+2\kappa')\log \frac{r_0}{r},
  $$
  which gives
  $$
  H(r)\leq C r^{n-1+2\kappa'}.
  $$
  This, however, contradicts the lower bound $H(r)\geq r^{n-1+2k}$.
  We must therefore have $\kappa\leq k$, and this establishes the second part of the
  lemma.

  \medskip 2) For the first part of the lemma we need to show that if
  $$
  \Phi_m(0+,u)=n-1+2\mu
  $$
  then
  $$
  \mu=\min\{\kappa, m\}.
  $$
  We consider two possibilities:

  \smallskip a) $\kappa<m$. Fix $\mu'\in (\kappa,m)$. Since
  $\kappa<k$, we must have $H(r)\geq r^{n-1+2k}$ for small $r$ and
  since $\kappa<\mu'$, we must have
  $$
  r\frac{H'(r)}{H(r)}\leq n-1+2\mu'
  $$
  for $0<r<r_0$. Integrating, we obtain that
  $$
  H(r)\geq H(r_0) \left(\frac{r}{r_0}\right)^{n-1+2\mu'}=c\,
  r^{n-1+2\mu'}
  $$
  for $0<r<r_0$. In particular, $H(r)> r^{n-1+2m}$ for sufficiently
  small $r$ and therefore
  $$
  \Phi_m(r,u)=(r+C_m r^2)\frac{H'(r)}{H(r)}=\frac{1+C_m r}{1+C_k
    r}\,\Phi_k(r,u).
  $$
  Hence $\mu=\kappa$ in this case.

  \smallskip b) $\kappa\geq m$. We need to show that $\mu=m$ in this
  case. In general, we know that $\mu\leq m$ from part 1) above,  so arguing by
  contradiction, assume $\mu<m$. Fix $\mu'\in (\mu,m)$. Then similarly
  to the arguments above, we will have
  $$
  r\frac{H'(r)}{H(r)}< n-1+2\mu'
  $$
  and consequently there exists $c>0$ such that
  $$
  H(r)>c\, r^{n-1+2\mu'},
  $$
  for small $0<r<r_0$. But then again $H(r)> r^{n-1+2m}\geq
  r^{n-1+2k}$ and therefore
  $$
  \Phi_m(r,u)=\frac{1+C_m r}{1+C_k r}\,\Phi_k(r,u)
  $$
  which again implies $\mu=\kappa$. However, as $\mu<m\leq \kappa$ this is not
  possible. This contradiction proves that $\mu=m$
  in this case.
\end{proof}

\begin{lemma}[Minimal and maximal
  frequency]\label{lem:min-max-homogen-nonzero} Let $u\in \S_k$ with
  $\Phi_k(0+,u)=n-1+2\kappa$, then one has
  $$
  2-\frac12\leq \kappa\leq k.
  $$
  Moreover, one has that either
  $$
  \kappa=2-\frac12\quad\text{or}\quad 2\leq \kappa\leq k.
  $$
\end{lemma}
\begin{proof} For $k=2$, it has been proved in \cite{CSS} that either $\kappa=2-\frac12$ or
  $\kappa\geq 2$. The same statement is also true for all $k\geq 2$
  from the identity
  $$
  \Phi_2(0+,u)=\min\{\Phi_k(0+,u),2\},
  $$
  which is a particular case of Lemma~\ref{lem:consist}
  relating values of $\Phi_k(0+,u)$ for different $k$. The upper bound
  is also contained in Lemma~\ref{lem:consist}.
\end{proof}

\begin{lemma}[Growth near the free boundary]\label{lem:growth-est} Let $u\in\S_k(M)$ and suppose that
  $\Phi_k(0+,u)=n-1+2\kappa$ with $\kappa\leq k$. Then
\begin{align*}
  H(r)&=\int_{\partial B_r} u^2\leq C_M r^{n-1+2\kappa}\\
  G(r)&=\int_{B_r} u^2\leq C_M r^{n+2\kappa}\\
  D(r)&=\int_{B_r} |\nabla u|^2\leq C_M r^{n-2+2\kappa}
\end{align*}
for $0<r<1/2$.
\end{lemma}
\begin{proof} We first prove the estimate for $H(r)$. The estimate is
  automatically satisfied for values $r$ such that $H(r)\leq
  r^{n-1+2k}$. Consider now a maximal open interval $(r_1,r_0)$ in
  $(0,1/2)$ where $H(r)>r^{n-1+2k}$. Then either $H(r_0)=r_0^{n-1+2k}$
  or $r_0=1/2$. In both cases $H(r_0)\leq M r_0^{n-1+2k}$.

  Further, we have
  $$
  (r+Cr^2)\frac{H'(r)}{H(r)}\geq n-1+2\kappa,\quad r_1<r<r_0.
  $$
  Dividing both sides by $r+Cr^2$ and integrating from $r$ to
  $r_0$, we obtain
\begin{align*}
  \log \frac{H(r_0)}{H(r)}&\geq (n-1+2\kappa)\int_{r}^{r_0}
  \frac{ds}{s(1+Cs)}\\
  &=(n-1+2\kappa)\log \frac{r_0/(1+Cr_0)}{r/(1+Cr)}.
\end{align*}
Exponentiation gives
$$
H(r)\leq C_M r^{n-1+2\kappa},\quad r_1<r<r_0.
$$
This proves the growth estimate for $H(r)$ and, after integration,
for $G(r)$.

To estimate $D(r)$ we note that $u$ satisfies the energy inequality
$$
\int_{B_{r/2}} |\nabla u|^2\leq \frac{C}{r^2}\int_{B_r^+} u^2+C
r^2\int_{B_r^+}(\Delta u)^2,
$$
which is proved exactly as the standard energy inequality in the
full ball $B_r$ (the boundary term on $B'_r$ vanishes since
$u\,\partial_{x_n} u=0$ there).  This gives
$$
D(r/2)\leq C r^{n+2\kappa-2}+C r^{n+2k}\leq C r^{n+2\kappa-2}.
$$
\end{proof}

\subsection{Optimal regularity}
Combining Lemmas~\ref{lem:min-max-homogen-nonzero} and
\ref{lem:growth-est} with the results of Caffarelli, Salsa, and
Silvestre \cite{CSS}, we obtain the following information about the
optimal (slowest possible) growth near the origin for a function
$u\in\S_k$.

\begin{proposition}[Optimal growth]\label{prop:optimal-growth} Let $u\in\S_k$, then
  $$
  \sup_{B_r} |u|\leq C\, r^{3/2}.\qed
  $$
\end{proposition}

Proposition \ref{prop:optimal-growth} in turn leads to the optimal
regularity of the solutions of the Signorini problem
\eqref{eq:signorini-1-nonzero}--\eqref{eq:0-fbp-nonzero}.

\begin{theorem}[Optimal regularity] Let $v\in\S^\phi$ with
  $\phi\in C^{2,1}(B_1)$, then $v\in C^{1,\frac12}_\loc(B_1^\pm\cup B_1')$.\qed
\end{theorem}

\section{Blowups}
\label{sec:blowups}

The generalized frequency formula in Theorem~\ref{thm:gen-almgren}
above allows to study the blowups. However, the situation is much
subtler than in the case of the zero obstacle. For $u\in\S_k$ and
$r>0$ we consider the rescalings $u_r$ introduced in
\eqref{eq:rescaling}. If it happens that $H(r)$ converges to $0$
faster than $r^{n-1+2k}$, then $\Phi_k(r,u)$ simply equals $(1+C_M
r)(n-1+2k)$ for small $r>0$, which does not help to control the
Dirichlet integral of $u_r$ on $B_1$. Thus, we don't know if the
blowups exist in that case. On the other hand,
the functional $\Phi_k(r,u)$ does contain enough
information to establish the uniform estimates for $u_r$ if
$\Phi_k(0+,u)< n-1+2k$.

\begin{lemma}[Uniform bounds of rescalings]\label{lem:resc-unif-bound}
  Let $u\in \S_k$ and suppose that $\Phi_k(0+,u)=n-1+2\kappa$ with
  $\kappa<k$. Then there exists a sufficiently small $r_0(u)>0$ such that the family $\{u_r\}_{0<r<r_0(u)}$ is uniformly bounded in
  $W^{1,2}(B_1)\cap C^1_\loc(B_1)$.
\end{lemma}

\begin{proof} We may assume that $H(r)>r^{n-1+2k}$ for $0<r<r_0$ and
  therefore the inequality $\Phi_k(r,u)\leq \Phi_k(r_0,u)$ will imply
  that
  $$
  r\,\frac{H'(r)}{H(r)}\leq C,\qquad 0<r<r_0.
  $$
  Using the formula for $H'(r)$ in Lemma~\ref{lem:deriv}, we have
  $$
  (n-1)+2r\,\frac{\int_{B_r}|\nabla u|^2+\int_{B_r} u\Delta
    u}{\int_{\partial B_r} u^2}\leq C,\qquad 0<r<r_0.
  $$
  From this inequality we obtain for the rescalings $u_r$
  $$
  \int_{B_1} |\nabla u_r|^2\leq C-r\frac{\int_{B_r} u\Delta
    u}{\int_{\partial B_r} u^2},\qquad 0<r<r_0.
  $$
  The second term in the right hand side can be controlled as
  follows.
\begin{align*}
  \left|\int_{B_r} u\Delta u\right|&\leq \left(\int_{B_r}
    u^2\right)^{1/2}\left(\int_{B_r\setminus B_r'}|\Delta u|^2
  \right)^{1/2}\leq C\,G(r)^{\frac12} r^{\frac{n}2+k-1}\\
  &\leq C r^{\frac{n}2+\kappa} r^{\frac{n}2+k-1}=C r^{n-1+\kappa+k}
\end{align*}
On the other hand, for any $\kappa'>\kappa$, we have that for
sufficiently small $r$
$$
\int_{\partial B_r} u^2> c\,r^{n-1+2\kappa'},
$$
for some $c>0$, see the arguments in the proof of
Lemma~\ref{lem:consist}. This implies that
$$
\left|\frac{\int_{B_r} u\Delta u}{\int_{\partial B_r}
    u^2}\right|\leq C r^{\kappa+k-2\kappa'}\leq C,
$$
provided we choose
$$
\kappa<\kappa'<\frac{\kappa+k}2.
$$
We thus obtain that
$$
\int_{B_1} |\nabla u_r|^2\leq C,\qquad 0<r<r_0.
$$
Together with
$$
\int_{\partial B_1} u_r^2=1,
$$
this gives the uniform boundedness of $\{u_r\}$ in $W^{1,2}(B_1)$.

Next, to see the boundedness in $C^1_\loc$, notice that
$$
|\Delta u_r(x)|=|\Delta u(rx)|\,
\frac{r^2r^{\frac{n-1}2}}{H(r)^{\frac12}}\leq M |x'|^{k-1}\frac{r^{k+1}r^{\frac{n-1}2}}{H(r)^{\frac12}}\leq
M\, r|x'|^{k-1}\quad\text{in } B_1\setminus B_1'
$$
if one uses the bound $H(r)>r^{n-1+2k}$. Thus, we obtain that $u_r$
is bounded in $C^{1,\alpha}_\loc(B_1^\pm\cup B_1')$, see e.g.\ \cite{Ca}.
\end{proof}

\begin{remark}
Although the extremal case $\kappa=k$ is not covered by
Lemma~\ref{lem:resc-unif-bound}, it should be considered as a
natural limitation that comes from having assumed that the thin
obstacle is in the class $C^{k,1}$. Indeed, if more regularity of
the thin obstacle were assumed, then the consistency
Lemma~\ref{lem:consist} would allow to study the blowups in the case
$\kappa=k$ as well.
\end{remark}

At this point, using Lemma~\ref{lem:resc-unif-bound} we see that,
under its assumptions, there exists a subsequence $r_j\to 0+$ such
that
\begin{equation}\label{eq:urj-u0-nonzero}
\begin{aligned}
  u_{r_j}\to u_0 &\quad\text{in }W^{1,2}(B_1)\\
  u_{r_j}\to u_0 &\quad\text{in }L^2(\partial B_1)\\
  u_{r_j}\to u_0 &\quad\text{in }C^1_\loc(B^\pm_1\cup B_1').
\end{aligned}
\end{equation}
We call such $u_0$ a \emph{blowup} of $u$ at the origin.

\begin{proposition}[Homogeneity of blowups]\label{prop:blowup-homogen-nonzero} Let $u\in\S_k$ and $\Phi_k(0+,u)=n-1+2\kappa$ with
  $\kappa<k$. Then every blowup $u_0$ is homogeneous of
  degree $\kappa$, $u_0\not \equiv 0$,  and satisfies \textup{\eqref{eq:signorini-1}--\eqref{eq:signorini-2}} \textup{(}i.e.\ $u_0$ solves the
  Signorini problem with zero obstacle\textup{)}.
\end{proposition}
\begin{proof} First notice that $u_0$ satisfies
\eqref{eq:signorini-1}--\eqref{eq:signorini-2}. Indeed, this follows
from from the $C^{1}_\loc$ convergence of $u_{r_j}\to u_0$ on $B_1^\pm\cup B_1'$ and
the estimate
$$
  |\Delta u_r (x)|\leq  M r|x'|^{k-1},\quad\text{in }B_1\setminus
  B_1'
  $$
that we established in the proof of Lemma~\ref{lem:resc-unif-bound}.

Next, since $\kappa<k$, the fact that
  $\Phi_k(0+,u)=n-1+2\kappa$ is equivalent to
  $$
  \lim_{r\to 0+} r\, \frac{H'(r)}{H(r)}= n-1+2\kappa.
  $$
  Arguing as in the proof of Lemma~\ref{lem:resc-unif-bound}, this
  relation
  can be reduced to
  $$
  \lim_{r\to 0+}\frac{r\int_{B_r}|\nabla u|^2}{\int_{\partial B_r}
    u^2}=\kappa,
  $$
  or in other words
  $$
  N(0+,u)=\kappa.
  $$
  But then we obtain for any $0<\rho<1$
  $$
  N(\rho, u_0)=\lim_{r_j\to 0+} N(\rho, u_{r_j})=\lim_{r_j\to 0+}
  N(\rho r_j, u)=\kappa.
  $$
This implies that $u_0$ is homogeneous of degree $\kappa$. Finally,
$u_0$ does not vanish identically since the convergence $u_{r_j}\to u_0$ in
$L^2(\partial B_1)$ and the equality $\int_{\partial B_1} u_{r_j}^2=1$
imply $\int_{\partial B_1} u_{0}^2=1$.
\end{proof}

\section{The free boundary}
\label{sec:free-boundary}

Suppose now we have a solution $v$ of the Signorini problem
\eqref{eq:signorini-1-nonzero}--\eqref{eq:0-fbp-nonzero} with
$\phi\in C^{k,1}(B')$. If $Q_k(x')$ and $\tilde Q_k(x)$ are the
$k$-th Taylor polynomial of $\phi$ and its symmetric harmonic
extension to $\R^n$, then
\begin{equation}\label{eq:uk}
u_k(x):=v(x)-\tilde Q_k(x)-(\phi(x')-Q_k(x'))\in \S_k.
\end{equation}
More generally, for $x_0\in\Gamma(v)$ we define
\begin{equation}\label{eq:ukx0}
u_k^{x_0}(x):=v(x+x_0)-\tilde Q_k^{x_0}(x)-(\phi(x'+x_0)-Q_k^{x_0}(x'))
\end{equation}
where $Q_k^{x_0}$ is the $k$-th Taylor polynomial of
$\phi(\cdot+x_0)$. The functions $u_{k}^{x_0}$ will satisfy the
conditions \eqref{eq:signorini-1-diff}--\eqref{eq:0-fbp-diff} but only
in a smaller ball $B_{1-|x_0|}$ instead of the full ball $B_1$. So
technically speaking $u_k^{x_0}$ are not in $\S_k$, however, most of
the time the results for class $\S_k$ can be used with insignificant
or no modification for functions $u_k^{x_0}$. In particular,
$\Phi_k(r, u_k^{x_0})$ will be monotone for $0<r<1-|x_0|$ and we can
use the value $\Phi_k(0+,u_k)$ to classify the free boundary points.

\begin{definition}\label{def:fb-class-nonzero} For $v\in\S^\phi$ we
  say that $0\in\Gamma^{(k)}_\kappa(v)$ for $2-\frac12\leq \kappa\leq
  k$ if and only if $\Phi_k(0+,u_k)=n-1+2\kappa$. More generally, we define
  $$
  \Gamma^{(k)}_\kappa(v):=\{x_0\in\Gamma(v) \mid
  \Phi_k(0+,u_k^{x_0})=n-1+2\kappa\}.
  $$
\end{definition}


\medskip
The next lemma shows how this classification of points changes with
$k$.

\begin{lemma}\label{lem:refinement-fb-class} Let $v\in\S^\phi$ with
  $\phi\in C^{k,1}(B_1')$. Then for any $2\leq m\leq k$ one has
\begin{align*}
  \Gamma^{(m)}_\kappa
  (v)&=\Gamma^{(k)}_\kappa(v)\quad\text{for }\kappa<m\\
  \Gamma^{(m)}_m(v)&=\bigcup_{\kappa\geq m}\Gamma^{(k)}_\kappa(v).
\end{align*}
\end{lemma}

\begin{proof} This lemma is a direct consequence of (in fact, it is equivalent to)
the consistency Lemma~\ref{lem:consist}. Let
  $u_k\in\S_k$ be as in \eqref{eq:uk} and
  $u_m\in\S_m$ be the function corresponding to an integer $2\leq
  m\leq k$. Then $u_m-u_k=o(|x|^m)$ and therefore
  $$
  \Phi_m(0+,u_m)=\Phi_m(0+,u_k)=\min\{\Phi_k(0+,u_k),n-1+2m\}.
  $$
  In fact, to establish the former equality, one has to argue
  similarly to the proof of Lemma~\ref{lem:consist} and just notice
  that
  $$
  \int_{\partial B_r} u_m^2> C\,r^{n-1+2 \mu'}\iff \int_{\partial
    B_r} u_k^2> C\, r^{n-1+2 \mu'}
  $$
  whenever $\mu'<m$. We leave the details to the reader.
\end{proof}

\begin{remark}
Thanks to Lemma~\ref{lem:refinement-fb-class} one can define the
sets
$$
\Gamma_\kappa(v):=\Gamma_\kappa^{(m)}(v),\quad\text{for
}\kappa<m\leq k,
$$
and the latter definition will not depend on the choice of
$m$. On the other hand, the classification obtained by using the
functional $\Phi_k$ could be viewed, loosely speaking, as a ``truncation'' of
a possibly finer classification of points. In particular, the set
$\Gamma^{(k)}_k$ can be considered as the bulk of points of
frequencies $\kappa\geq k$. More precisely, if one knows higher
$C^{k',1}$ regularity of the thin obstacle $\phi$ with $k'>k$, then
$\Gamma^{(k)}_k$ is refined into the union of $\Gamma^{(k')}_\kappa$
with $k\leq \kappa\leq k'$.
\end{remark}

As we have already mentioned in the case of the zero thin obstacle,
the sets $\Gamma_\kappa(v)$ are nonempty only for specific values of
$\kappa$, see Remarks~\ref{rem:missing_values} and
\ref{rem:poss_freq_dim2}. In higher dimensions the only information
known is the one contained in
Lemma~\ref{lem:min-max-homogen-nonzero}, i.e.
$$
\kappa=2-\frac12,\quad\text{or}\quad \kappa\geq 2.
$$

\begin{definition}[Regular points] For  $v\in\S^\phi$ with $\phi\in C^{2,1}(B_1')$ the free boundary point $x_0\in\Gamma(v)$ is called \emph{regular} if
  $x_0\in\Gamma_{2-\frac12}(v)$.
\end{definition}

It is easy to see that the mapping $x_0\mapsto \Phi_k(0+,
u_k^{x_0})$ is upper semicontinuous, and since the value
$\kappa=2-\frac12$ is isolated, one immediately obtains that
$\Gamma_{2-\frac12}(v)$ is a relatively open subset of $\Gamma(v)$.
Furthermore, the following theorem has been established by
Caffarelli, Salsa, and Silvestre \cite{CSS}.

\begin{theorem}[Regularity of the regular set] Let $v\in\S^\phi$ with
  $\phi\in C^{2,1}$. Then $\Gamma_{2-\frac12}(v)$ is locally a
  $C^{1,\alpha}$-regular $(n-2)$-dimensional surface.
\end{theorem}

\section{Singular set: statement of main results}
\label{sec:singular-set:-main-1}

Similarly to the case of the zero obstacle, in this section we state
our main results on the structure of the singular set of the free
boundary. The proofs will be given in Section~\ref{sec:singular-set:-proofs}. We begin with the relevant
definition.

\begin{definition}[Singular points]\label{def:sing_pts_nonzero} Let $v\in\S^\phi$. We say
  that $x_0\in\Gamma(v)$ is a \emph{singular free boundary point} if
  $$
  \lim_{r\to 0+}\frac{\H^{n-1}(\Lambda(v)\cap B_r')}{\H^{n-1}(B_r')}=0.
  $$
  We denote the set of singular points by $\Sigma(v)$. If the thin
  obstacle $\phi$ is $C^{k,1}$ regular and $\kappa<k$ then we also
  define
  $$
  \Sigma_\kappa(v):=\Gamma_\kappa(v)\cap \Sigma(v).
  $$
\end{definition}

It will be convenient to abuse the notation and write
$0\in\Sigma_\kappa(u)$ for $u\in\S_k$, whenever $\Lambda(u)$
satisfies a vanishing condition similar to that for $\Lambda(v)$ in
Definition~\ref{def:sing_pts_nonzero} above.

We start with a characterization of singular points in terms of blowups at the generalized frequency. In particular, we show that
$$
\Sigma_\kappa(v)=\Gamma_\kappa(v),\quad\text{for }\kappa=2m<k,\ m\in \N.
$$

\begin{theorem}[Characterization of singular points]\label{thm:blowup-sing-nonzero}
  Let $u\in\S_k$ and $0\in \Gamma_\kappa(u)$ for $\kappa<k$.
 Then the following statements are equivalent:
\begin{enumerate}
\item[(i)] $0\in \Sigma_\kappa(u)$
\item[(ii)] any blowup of $u$ at the origin is a nonzero homogeneous
 polynomial $p_\kappa$ of degree $\kappa$ from the class $\P_\kappa$, i.e.
 $$
 \Delta p_\kappa=0,\quad p_\kappa(x',0)\geq 0,\quad
 p_\kappa(x',-x_n)=p_\kappa(x',x_n)
 $$
\item[(iii)] $\kappa=2m$ for some $m\in\N$.
\end{enumerate}
\end{theorem}
\begin{proof} The proof is a minor modification of that of
  Theorem~\ref{thm:blowup-sing} and is therefore omitted.
\end{proof}

The next result is the key step in the study of the singular set.

\begin{theorem}[$\kappa$-differentiability at singular points]\label{thm:k-diff-sing-p-nonzero} Let
  $u\in\S_k$ and $0\in\Sigma_\kappa(u)$ for $\kappa=2m<k$, $m\in\N$.
  There exists a nonzero $p_\kappa\in\P_\kappa$ such that
  $$
  u(x)=p_\kappa(x)+o(|x|^\kappa).
  $$
  Moreover, if $v\in\S^\phi$ with $\phi\in C^{k,1}(B_1')$,
  $x_0\in\Sigma_\kappa(v)$ and $u^{x_0}_k$ is obtained as in
  \eqref{eq:ukx0}, then in the Taylor expansion
  $$
  u^{x_0}_k(x)=p_\kappa^{x_0}(x)+o(|x|^\kappa)
  $$
  the mapping $x_0\mapsto p^{x_0}_\kappa$ from $\Sigma_\kappa(v)$
  to $\P_\kappa$ is continuous.
\end{theorem}

\begin{definition}[Dimension at the singular point]\label{def:dim-sing-point-nonzero} For  $v\in\S^\phi$ and a singular point $x_0\in\Sigma_k(v)$ we denote
  $$
  d_\kappa^{x_0}:=\dim\{\xi\in\R^{n-1} \mid \xi\cdot \nabla_{x'}
  p_\kappa^{x_0}(x',0)=0\text{ for every }x'\in\R^{n-1} \}
  $$
  the degree of degeneracy of the polynomial $p_\kappa^{x_0}$,
  which we call the dimension of $\Sigma_\kappa(v)$ at $x_0$. Note that
  since $p_\kappa^{x_0}\not\equiv 0$ on $\R^{n-1}\times\{0\}$ one has
  $$
  0\leq d_\kappa^{x_0}\leq n-2.
  $$
  Then for $d=0,1,\ldots, n-2$ define
  $$
  \Sigma_\kappa^d(v):=\{x_0\in\Sigma_\kappa(u) \mid d^{x_0}_\kappa=d\}.
  $$
\end{definition}

\begin{theorem}[Structure of the singular set]\label{thm:sing-points-nonzero} Let
  $v\in\S^\phi$ with $\phi\in C^{k,1}(B_1')$. Then
 $\Sigma_\kappa(v)=\Gamma_\kappa(v)$  for $\kappa=2m<k$, $m\in\N$, and every set 
  $\Sigma_\kappa^d(v)$ for $d=0,1,\ldots, 
  n-2$, is contained in a countable union of $d$-dimensional $C^1$
  manifolds.
\end{theorem}

\section{Weiss and Monneau type  monotonicity formulas}
\label{sec:weiss-monneau-type}

The main tools in the proof of Theorems
\ref{thm:k-diff-sing-p-nonzero} and \ref{thm:sing-points-nonzero}
are Weiss and Monneau type monotonicity formulas, similar to those
in Section~\ref{sec:weiss-type-monot}.

\subsection{Weiss type monotonicity formulas}
\label{sec:weiss-type-monot-2}

\begin{theorem}[Weiss type Monotonicity Formula]\label{thm:Weiss-monot-nonzero}
  Let $u\in\S_k(M)$ and suppose that $0\in \Gamma^{(k)}_\kappa(u)$ for $\kappa\leq k$. There exist $r_M>0$ and $C_M\geq 0$ such that
\begin{align*}
  W_\kappa(r,u)&:=\frac{1}{r^{n-2+2\kappa}}\int_{B_r}|\nabla u|^2-\frac{\kappa}{r^{n-1+2 \kappa}}\int_{\partial B_r} u^2\\
  &=\frac{1}{r^{n-2+2 \kappa}} D(r)-\frac{\kappa}{r^{n-1+2 \kappa}}
  H(r).
\end{align*}
satisfies
$$
\frac{d}{dr}W_\kappa(r)\geq - C_M\quad\text{for }0<r<r_M.
$$
\end{theorem}

\begin{proof}
  Proof by a direct computation. Using Lemmas~\ref{lem:deriv} and
  \ref{lem:growth-est}, we obtain
\begin{align*}
  \frac{d}{dr}W_\kappa(u,r)&=\frac{1}{r^{n-2+2\kappa}}\left\{D'(r)-\frac{n-2+2\kappa}r
    D(r)-\frac\kappa r H'(r)+\frac{\kappa(n-1+2\kappa)}{r^2}H(r)\right\}\\
  &=\frac{2}{r^{n-2+2\kappa}}\left\{\int_{\partial B_r}(\partial_\nu
    u)^2-\frac{\kappa}{r} D(r)-\frac{\kappa}{r}\int_{\partial B_r}
    u\partial_\nu u+\frac{\kappa^2}{r^2}\int_{\partial B_r}
    u^2\right.\\
  &\phantom{\mbox{}=\frac{2}{r^{n-2+2\kappa}}\bigg\{}\left.\mbox{}-\frac1r\int_{B_r}\Delta u (\nabla u\cdot x)\right\}\\
  &=\frac{2}{r^{n-2+2\kappa}}\left\{\int_{\partial B_r}(\partial_\nu
    u)^2-\frac{2\kappa}{r}\int_{\partial B_r} u\partial_\nu
    u+\frac{\kappa^2}{r^2}\int_{\partial B_r}
    u^2\right.\\
  &\phantom{\mbox{}=\frac{2}{r^{n-2+2\kappa}}\bigg\{}\left.\mbox{}+\frac{\kappa}r\int_{B_r}
    u\Delta u-\frac1r\int_{B_r}\Delta u
    (\nabla u\cdot x)\right\}\\
  &=\frac{2}{r^{n-2+2\kappa}}\left\{ \frac1{r^2}\int_{\partial B_r}
    (x\cdot \nabla u-\kappa u)^2-\frac1r\int_{B_r} \Delta u
    (x\cdot\nabla u-\kappa u)
  \right\}\\
  &\geq -\frac{2}{r^{n-1+2\kappa}}\int_{B_r} \Delta u (x\cdot\nabla
  u-\kappa u)\\
  &\geq-\frac{C}{r^{n-1+2\kappa}}\left(\int_{B_r^+}(\Delta
    u)^2\right)^{\frac12}\left(\int_{B_r^+}(x\cdot\nabla u-\kappa
    u)^2\right)^{\frac12}\\
  &\geq
  -C\frac{r^{\frac{n}2+k-1}r^{\frac{n}2+\kappa}}{r^{n-1+2\kappa}}=-C
  r^{k-\kappa}\geq -C,
\end{align*}
where in the step next to the last we have used the estimates in Lemma~\ref{lem:growth-est}.
\end{proof}


\subsection{Monneau type monotonicity formulas}
\label{sec:monn-type-monot-1}

\begin{theorem}[Monneau type Monotonicity Formula]\label{thm:Monn-formula-nonzero} Let $u\in\S_k(M)$ and
  suppose that $0\in\Sigma_\kappa(u)$ with $\kappa=2m<k$, $m\in\N$.
  For any $p_\kappa\in\P_\kappa$ there exist $r_M>0$ and $C_M\geq 0$
  such that
  $$
  M_\kappa(r,u,p_\kappa)=\frac{1}{r^{n-1+2\kappa}}\int_{\partial
    B_r} (u-p_\kappa)^2
  $$
  satisfies
  $$
  \frac{d}{dr} M_\kappa(r,u,p_\kappa)\geq -C_M
  \left(1+\|p_\kappa\|_{L^2(B_1)}\right)\quad\text{for }0<r<r_M.
  $$
\end{theorem}
\begin{proof} First note that if $0\in \Gamma_\kappa^{(k)}(u)$ for
  $\kappa<k$, then
  $$
  W_\kappa(0+,u)=0.
  $$
  Indeed, using Lemma~\ref{lem:deriv}, we represent
\begin{align*}
  W_\kappa(r,u)&=\frac1{2\,r^{n-1+2 \kappa}}\left(r H'(r)-(n-1+2 \kappa) H(r)-2r\int_{B_r} u \Delta u\right)\\
  &=\frac{H(r)}{2\,r^{n-1+2 \kappa}}\left(r\frac{H'(r)}{H(r)}-(n-1+2
    \kappa)\right)-\frac{\int_{B_r} u \Delta u}{r^{n-2+2 \kappa}}.
\end{align*}
Now the identity $W_\kappa(0+,u)=0$ follows from the following facts:
\begin{itemize}
\item[(i)] $r\dfrac{H'(r)}{H(r)}\to n-1+2 \kappa$, since $\kappa<k$,

\item[(ii)]$\dfrac{H(r)}{r^{n-2+2 \kappa}}$ is bounded, in view of the growth
  estimate in Lemma~\ref{lem:growth-est},

\item[(iii)] $\dfrac{\left|\int_{B_r} u \Delta u\right|}{r^{n-1+2
      \kappa}}\leq C\, r^{k+1-\kappa}$, by the arguments in the proof of
  Lemma~\ref{lem:resc-unif-bound}.
\end{itemize}
Next, we also observe that $W_\kappa(r,p_\kappa)=0$ for any
$p_\kappa\in\P_\kappa$. Setting $w=u-p_\kappa$, and repeating the
the computations in the proof of Theorem~\ref{thm:Monn-mon-form}, we
then obtain
$$
W_\kappa(r,u)= \frac{1}{r^{n-2+2\kappa}}\int_{B_r}{(-w\Delta
  w)}+\frac{1}{r^{n-1+2\kappa}}\int_{\partial B_r} w(x\nabla w-\kappa
w)
$$
and
$$
\frac{d}{dr}M_\kappa(r,u,p_\kappa)=\frac{2}{r^{n+2\kappa}}\int_{\partial
  B_r} w (x\cdot\nabla w-\kappa w).
$$
This gives
$$
\frac{d}{dr}M_\kappa(r,u,p_\kappa)=\frac{2}{r}\,
W_\kappa(r,u)+\frac{2}{r^{n-1+2\kappa}}\int_{B_r} w\Delta w.
$$
For the first term in the right-hand side note that by
Theorem~\ref{thm:Weiss-monot-nonzero} we obtain
$$
W_\kappa(r,u)=W_\kappa(r,u)-W_\kappa(0+,u)\geq -C_M r.
$$
For the second term, note that $w\Delta w$ is a nonnegative measure
on $B'_r$. Off $B'_r$, we can control $w\Delta w=w\Delta u$ by the
Cauchy-Schwarz inequality. Hence,
\begin{align*}
  \frac{d}{dr}M_\kappa(r,u,p_\kappa)&\geq
  -C_M-\frac{2}{r^{n-1+2\kappa}}\left(\int_{B_r^+}
    w^2\right)^{\frac12}\left(\int_{B_r^+}(\Delta u)^2\right)^{\frac12}\\
  &
  \geq -C_M-C_M\left(C_M+\|p_\kappa\|_{L^2(B_1)}\right) r^{k-\kappa}\\
  &\geq -C_M\left(1+\|p_\kappa\|_{L^2(B_1)}\right).
\end{align*}
\end{proof}

\section{Singular set: proofs}
\label{sec:singular-set:-proofs}

In this section we prove Theorems~\ref{thm:k-diff-sing-p-nonzero}
and \ref{thm:sing-points-nonzero}. The structure of the proofs is
essentially the same as for their counterparts in the zero obstacle
case, see Theorems~\ref{thm:k-diff-sing-p} and
\ref{thm:sing-points}. The only difference is that they based on the
results that have been so far developed for the case of a nonzero
obstacle. We start with the nondegeneracy lemma, similar to
Lemma~\ref{lem:nondeg-sing} in the zero obstacle case.

\begin{lemma}[Nondegeneracy at singular points]\label{lem:nondeg-sing-nonzero} Let $u\in\S_k$ and
  $0\in\Sigma_\kappa (u)$ for $\kappa<k$. There exists $c>0$ such that
  $$
  \sup_{\partial B_r} |u(x)|\geq c\,r^\kappa.
  $$
\end{lemma}
\begin{proof} Assume the contrary. Then for a sequence $r=r_j\to 0$
  one has
  $$
  h_r:=\left(\frac{1}{r^{n-1}}\int_{\partial B_r}
    u^2\right)^{1/2}=o(r^{\kappa}).
  $$
  Passing to a subsequence if necessary we may assume that
  $$
  u_r(x)=\frac{u(rx)}{h_r}\to q_\kappa(x)\quad\text{uniformly on
  }\partial B_1
  $$
  for some nonzero $q_\kappa\in\P_\kappa$.  Now for such $q_\kappa$
  we apply Theorem~\ref{thm:Monn-formula-nonzero}
  to $M_\kappa(r,u,q_\kappa)$. From the assumption on the growth of $u$ is is easy to
  recognize that
  $$
  M_\kappa(0+,u, q_\kappa)=\int_{\partial B_1}
  q_\kappa^2=\frac{1}{r^{n-1+2\kappa}} \int_{\partial B_r} q_\kappa^2
  $$
  Therefore, using the monotonicity of $M(r,u,q_\kappa)+C\,r$ (see
  Theorem~\ref{thm:Monn-formula-nonzero}) for appropriately chosen
  $C>0$, we will have that
  $$
  C\,r+\frac{1}{r^{n-1+2\kappa}} \int_{\partial B_r} (u-q_\kappa)^2\geq
  \frac{1}{r^{n-1+2\kappa}} \int_{\partial B_r} q_\kappa^2
  $$
  or equivalently
  $$
  \frac{1}{r^{n-1+2\kappa}}\int_{\partial B_r} u^2-2 u q_\kappa\geq -C\,r.
  $$
  After rescaling, we obtain
  $$
  \frac{1}{r^{2\kappa}}\int_{\partial B_1} h_r^2 u_r^2-2h_r r^{\kappa}
  u_r q_\kappa\geq - C\,r,
  $$
 which we can rewrite as
  $$
  \int_{\partial B_1} \frac{h_r}{r^\kappa} u_r^2-2u_r q_\kappa\geq
  -C\,\frac{r^{\kappa+1}}{h_r}.
  $$
Now from the arguments in the proof of Lemma~\ref{lem:consist} we have $H(r)> c\,r^{n-1+2\kappa'}$ for any $\kappa'>\kappa$ and if we choose
$\kappa'<\kappa+1$ we will have that $r^{\kappa+1}/h_r\to 0$. Thus, passing to
the limit over $r=r_j\to 0$, we arrive at
  $$
  -\int_{\partial B_1} q_\kappa^2\geq 0,
  $$
which is a contradiction as $q_\kappa\not=0$.
\end{proof}

\begin{lemma}[$\Sigma_k(v)$ is $F_\sigma$]\label{lem:sing-set-F-sigma-nonzero} For any $v\in \S^\phi$ with $\phi\in C^{k,1}(B_1')$, the set $\Sigma_\kappa(v)$ with $\kappa=2m<k$, $m\in \N$, is of type $F_\sigma$, i.e., it is a union of countably many closed sets.
\end{lemma}
\begin{proof} As in the zero-obstacle case (see Lemma~\ref{lem:sing-set-F-sigma}) we show that $\Sigma_\kappa(v)$ is the union of sets $E_j$ of points $x_0\in \Sigma_\kappa(v)\cap \overline{B_{1-1/j}}$ satisfying
\begin{equation}\label{eq:Ej-nonzero}
\tfrac1j\, \rho^{\kappa}\leq \sup_{|x|=\rho} |u^{x_0}_k(x)| < j \rho^\kappa
\end{equation}
for $0<\rho<1-|x_0|$. The proof that $E_j$ are closed is almost identical to that in Lemma~\ref{lem:sing-set-F-sigma}.
\end{proof}

\begin{theorem}[Uniqueness of the homogeneous blowup at singular
  points]\label{thm:uniq-blowup-sing-nonzero} Let $u\in\S_k$ and
  $0\in\Sigma_\kappa(u)$ with $\kappa<k$. Then there exists a unique nonzero
  $p_\kappa\in\P_\kappa$ such that
  $$
  u_r^{(\kappa)}(x):=\frac{u(rx)}{r^\kappa}\to p_\kappa(x).
  $$
\end{theorem}

\begin{proof} Let $u^{(\kappa)}_r(x)\to u_0(x)$ in
  $C^{1,\alpha}_{\loc}(\R^n)$ over a certain subsequence $r=r_j\to
  0+$. The existence of such limit follows from the growth estimate
  $|u(x)|\leq C|x|^\kappa$. Moreover, the nondegeneracy implies that
  $u_0$ in not identically zero. Next, we have that
  $$
  W_\kappa(r, u_0)=\lim_{r_j\to 0+} W_\kappa(r,
  u_{r_j}^{(\kappa)})=\lim_{r_j\to 0+} W_\kappa(rr_j,
  u)=W_\kappa(0+,u)=0,
  $$
  for any $r>0$, implying that $u_0$ is homogeneous of degree
  $\kappa$ solution of
  \eqref{eq:signorini-1}--\eqref{eq:signorini-2}. Repeating the
  arguments in
  Theorem~\ref{thm:blowup-sing-nonzero}, we see that $u_0$ must
  be a polynomial in $\P_\kappa$.

  We now apply Theorem~\ref{thm:Monn-formula-nonzero} to the pair $u$,
  $u_0$. We obtain that the limit
  $M_\kappa(0+,u,u_0)$ exists and can be computed by
  $$
  M_\kappa(0+, u, u_0)=\lim_{r_j\to 0+} M_\kappa(r_j, u,
  u_0)=\lim_{j\to\infty} \int_{\partial B_1} (u^{(\kappa)}_{r_j}-u_0)^2=0.
  $$
  In particular, we obtain that
  $$
  \int_{\partial B_1} (u^{(\kappa)}_r(x)-u_0)^2= M_\kappa(r,
  u,u_0)\to 0
  $$
  as $r\to 0+$ (not just over $r=r_j\to 0+$!). Thus, if $u_0'$ is a
  limit of $u_r^{(\kappa)}$ over another sequence $r=r_j'\to
  0$, we conclude that
  $$
  \int_{\partial B_1} (u_0'-u_0)^2=0.
  $$
  Since both $u_0$ and $u_0'$ are homogeneous of degree $\kappa$,
  they must coincide.
\end{proof}

\begin{theorem}[Continuous dependence of blowup]
\label{thm:cont-dep-blowup-nonzero} Let $v\in\S^\phi$
  with $\phi\in C^{k,1}(B_1')$. For
  $x_0\in\Sigma_\kappa(v)$ let $u_k^{x_0}$ be as in \eqref{eq:ukx0}
  and denote by
  $p^{x_0}_\kappa$ the blowup of
  $u^{x_0}_k$ at $x_0$ as in Theorem~\ref{thm:uniq-blowup-sing-nonzero} so that
  $$
  u^{x_0}_k(x)=p^{x_0}_\kappa(x)+o(|x|^\kappa).
  $$
  Then the mapping $x_0\mapsto p_\kappa^{x_0}$ from
  $\Sigma_\kappa(v)$ to $\P_\kappa$ is continuous. Moreover, for any compact subset $K$ of $\Sigma_\kappa(v)\cap B_1$ there
  exists a modulus of continuity $\sigma=\sigma^K$, $\sigma(0+)=0$ such that
  $$
  |u^{x_0}_k(x)-p^{x_0}_\kappa(x)|\leq \sigma (|x|)|x|^\kappa
  $$
  for any $x_0\in K$.
\end{theorem}


\begin{proof} The proof is similar to that of Theorem~\ref{thm:cont-dep-blowup}. Given $x_0$ and  $\epsilon>0$ fix $r_\epsilon=r_\epsilon(x_0)>0$ such that
  $$
  M_\kappa(r_\epsilon, u^{x_0}_k,
  p_\kappa^{x_0}):=\frac{1}{r_\epsilon^{n-1+2\kappa}}\int_{\partial
    B_{r_\epsilon}} (u^{x_0}_k(x)-p^{x_0}_\kappa)^2<\epsilon.
  $$
  Then there exists $\delta_\epsilon=\delta_\epsilon(x_0)$ such that if
  $x_0'\in\Sigma_\kappa(u)$ and $|x_0'-x_0|<\delta_\epsilon$ then
  $$
  M_\kappa(r_\epsilon, u^{x_0'}_k,
  p_\kappa^{x_0})=\frac{1}{r_\epsilon^{n-1+2\kappa}}\int_{\partial
    B_{r_\epsilon}} (u^{x_0'}_k-p^{x_0}_\kappa)^2<2\epsilon.
  $$
This follows from the continuous dependence of $u^{x_0}_k$ on
$x_0\in\Gamma(v)$, which in turn is a consequence of $C^k$
differentiability of the thin obstacle $\phi$.

  From Theorem~\ref{thm:Monn-formula-nonzero}, we will have that
  $$
  M_\kappa(r, u^{x_0'}_k, p_\kappa^{x_0})< 2\epsilon+ C\,r_\epsilon,\quad
  0<r<r_\epsilon
  $$
for a constant $C=C(x_0)$ depending on $L^2$ norms of $u^{x_0'}_k$ and
$p_\kappa^{x_0}$, which can be made uniform for $x_0'$ in a small
neighborhood of $x_0$ as $u_k^{x_0'}$ depends continuously on $x_0'$.
  Passing $r\to 0$ we will therefore obtain
  $$
  M(0+,u^{x_0'}_k,p_\kappa^{x_0})=\int_{\partial B_1}
  (p_\kappa^{x_0'}-p_\kappa^{x_0})^2\leq 2\epsilon+C\,r_\epsilon.
  $$
  This shows the first part of the theorem.

  To show the second part, we notice that we have
  \begin{align*}
  \|u^{x_0'}_k-p^{x_0'}_\kappa\|_{L^2(\partial B_r)}&\leq
  \|u^{x_0'}_k-p^{x_0}_\kappa\|_{L^2(\partial B_r)}
  +\|p^{x_0'}_\kappa-p^{x_0'}_\kappa\|_{L^2(\partial B_r)}\\
 &\leq
  2(2\epsilon +C\,r_\epsilon)^{\frac12}r^{\frac{n-1}2+\kappa},
\end{align*}
for $|x_0'-x_0|<\delta_\epsilon$, $0<r<r_\epsilon$, or equivalently
\begin{equation}\label{eq:w-p-L2-nonzero}
 \|w^{x_0'}_r-p^{x_0'}_\kappa\|_{L^2(\partial B_1)}\leq 2(2\epsilon+C r_\epsilon)^{\frac12},
\end{equation}
where
$$
w^{x_0'}_r(x):=\frac{u^{x_0'}_k(rx)}{r^\kappa}.
$$
Making a finite cover of the compact $K$ with balls $B_{\delta_\epsilon(x_0^i)}(x_0^i)$ for some $x_0^i\in K$, $i=1,\ldots,N$, we see that \eqref{eq:w-p-L2-nonzero} is satisfied for all $x_0'\in K$, $r<r_\epsilon^K:=\min\{r_\epsilon(x_0^i)\mid i=1,\ldots,N\}$ and $C=C^K:=\max\{C(x_0^i)\mid i=1,\ldots, N\}$.

Now notice that
$$
w^{x_0'}_r\in\S_k,\quad
p_\kappa^{x_0'}\in \S_k
$$
with $C^{1,\alpha}(\overline {B_1})$ norms of $w^{x_0'}_r$ and
$p_\kappa^{x_0'}$ uniformly
bounded and the estimate
$$
|\Delta w^{x_0'}_r|\leq M r^{k-\kappa}\to 0,\quad\text{a.e.\ in }B^\pm_1.
$$
Thus, using a compactness argument, we can show that
$$
\|w^{x_0'}_r-p^{x_0'}_\kappa\|_{L^\infty(B_{1/2})}\leq C_\epsilon,
$$
for all $x_0'\in K$, $r<r_\epsilon^K$ and $C_\epsilon\to
0$ as $\epsilon\to 0$. It is now easy to
  see that this implies the second part of the theorem.
\end{proof}

At this point we are ready to present the proofs of the main
results.

\medskip

\begin{proof}[Proof of Theorem~\ref{thm:k-diff-sing-p-nonzero}] It is
obtained by combining Theorems~\ref{thm:uniq-blowup-sing-nonzero}
and \ref{thm:cont-dep-blowup-nonzero} above.
\end{proof}

\begin{proof}[Proof of Theorem~\ref{thm:sing-points-nonzero}] The
  proof is almost the same as for
  Theorem~\ref{thm:sing-points}, but is now based on
  Theorem~\ref{thm:k-diff-sing-p-nonzero} instead of
  Theorem~\ref{thm:k-diff-sing-p}.

We give few details. For any
  $x_0\in\Sigma_\kappa(v)$ let the polynomial
  $p^{x_0}_\kappa\in\P_\kappa$ be as
  in Theorem~\ref{thm:cont-dep-blowup-nonzero}. Write it in the
  expanded form
  $$
  p^{x_0}_\kappa(x)=\sum_{|\alpha|=\kappa}
  \frac{a_\alpha(x_0)}{\alpha!}x^\alpha.
  $$
  Then the coefficients $a_\alpha(x)$ are continuous on
  $\Sigma_{\kappa}(v)$. Moreover, if
  $x\in \Sigma_\kappa(v)$, we have $x\in B_1'$ and therefore  $\tilde
  Q_k^{x_0}(x)=Q_k^{x_0}(x)$ in the definition \eqref{eq:ukx0}, which
  implies that
$$
u_k^{x_0}(x-x_0)=v(x)-\phi(x)=0.
$$
Hence, from Theorem~\ref{thm:cont-dep-blowup-nonzero} we obtain
  $$
  |p^{x_0}_\kappa(x-x_0)|\leq \sigma(|x-x_0|)|x-x_0|^\kappa,\quad
  x, x_0\in K.
  $$
for a compact subset $K\subset \Sigma_\kappa(v)$. Furthermore, if we define
  $$
  f_\alpha(x)=\begin{cases} 0 & |\alpha|<\kappa\\ a_\alpha(x) &
    |\alpha|=\kappa,
\end{cases}\qquad x\in\Sigma_k(v),
$$
as in the zero-obstacle case, then we have the following compatibility lemma .

\begin{lemma}\label{lem:Whitney-compat-nonzero}
Let $K=E_j$ as in Lemma~\ref{lem:sing-set-F-sigma-nonzero}. Then for any $x_0,x\in K$
\begin{equation}\label{eq:compat-1-nonzero}
f_\alpha(x)=\sum_{|\beta|\leq
  \kappa-|\alpha|}\frac{f_{\alpha+\beta}(x_0)}{\beta!}(x-x_0)^\beta +R_\alpha(x,x_0)
\end{equation}
with
\begin{equation}\label{eq:compat-2-nonzero}
|R_\alpha(x,x_0)|\leq \sigma_\alpha(|x-x_0|)|x-x_0|^{\kappa-|\alpha|},
\end{equation}
where $\sigma_\alpha=\sigma_\alpha^K$ is a certain modulus of continuity.
\end{lemma}
\begin{proof} There are few additional details compared to Lemma~\ref{lem:Whitney-compat}, so we give a complete proof.
	
1) In the case $|\alpha|=\kappa$ we have
$$
R_\alpha(x,x_0)=a_\alpha(x)-a_\alpha(x_0)
$$
and the statement follows from the continuity of $a_\alpha(x)$ on $\Sigma_k(v)$.

\smallskip
2) For $0\leq |\alpha|<\kappa$ we have

$$
R_\alpha(x,x_0)=-\sum_{\gamma>\alpha, |\gamma|=\kappa} \frac{a_\gamma(x_0)}{(\gamma-\alpha)!}(x-x_0)^{\gamma-\alpha}= - \partial^\alpha p^{x_0}_\kappa (x-x_0).
$$
Now suppose that there exists no modulus of continuity $\sigma_\alpha$ such that \eqref{eq:compat-2-nonzero} is satisfied for all $x_0, x\in K$. Then there exists $\delta>0$ and a sequence $x_0^i, x^i\in K$ with 
$$
|x^i-x_0^i|=:\rho_i\to 0
$$
such that
\begin{equation}\label{eq:compat-contr-nonzero}
\Big|\sum_{\gamma>\alpha, |\gamma|=\kappa} \frac{a_\gamma(x_0^i)}{(\gamma-\alpha)!}(x^i-x_0^i)^{\gamma-\alpha}\Big|\geq \delta |x^i-x_0^i|^{\kappa-|\alpha|}.
\end{equation}
Consider the rescalings
$$
w^i(x)=\frac{u^{x_0^i}_k(\rho_i x)}{\rho_i^\kappa},\quad \xi^i=(x^i-x_0^i)/\rho_i.
$$
Without loss of generality we may assume that $x_0^i\to x_0\in K$ and $\xi^i\to \xi_0\in \partial B_1$.
From Theorem~\ref{thm:cont-dep-blowup-nonzero} we have that
$$
|w^i(x)-p^{x_0^i}_\kappa(x)|\leq \sigma(\rho_i|x|)|x|^\kappa
$$
and therefore
\begin{equation}\label{eq:wip}
	w^i(x)\to p^{x_0}_\kappa(x)\quad\text{in } L^\infty_{\loc}(\R^n).
\end{equation}
We also consider the rescalings at $x^i$ instead of $x_0^i$
$$
\tilde w^i(x)=\frac{u^{x^i}_k(\rho_i x)}{\rho_i^\kappa}.
$$
We then claim that the $C^{k,1}$ regularity of the thin obstacle $\phi$ implies that 
\begin{equation}\label{eq:wixi}
w^i(x+\xi^i)-\tilde w^i(x)\to 0\quad\text{in } L^\infty_{\loc}(\R^n).
\end{equation}
Indeed, if $Q^{x_0}_k(x')$ denotes the $k$-th Taylor polynomial of $\phi(x')$ at $x_0$, then
\begin{align*}
&\frac{Q^{x_0^i}_k(\rho_i(x'+\xi_i))-Q^{x^i}_k(\rho_i x')}{\rho_i^\kappa}\\&=\frac{\phi(x_0^i+\rho_i(x'+\xi_i))+o(\rho_i^k|x'+\xi_i|^k)-\phi(x^i+\rho_ix')-o(\rho_i^k|x'|^k)}{\rho_i^\kappa}\\
&=o(\rho_i^{k-\kappa})\to 0
\end{align*}
and this implies the convergence \eqref{eq:wixi}, if we write the explicit definition of $w^i$ using \eqref{eq:ukx0}. Further, note that since $x^i\in E_j$, we have 
$$
\tfrac1j \rho^\kappa\leq \sup_{|x|=\rho} |u^{x^i}_k(x)|\leq j \rho^\kappa
$$
and therefore
$$
\tfrac1j \rho^\kappa\leq \sup_{|x|=\rho} |\tilde w^i(x)|\leq j \rho^\kappa.
$$
Passing to the limit in \eqref{eq:wip}--\eqref{eq:wixi} we obtain that
$$
\tfrac1j \rho^\kappa\leq \sup_{|x|=\rho} p^{x_0}_\kappa(x+\xi_0)\leq j \rho^\kappa,\quad 0<\rho<\infty.
$$
This implies that $\xi_0$ is a point of frequency $\kappa=2m$ for the polynomial $p^{x_0}_\kappa$ and by Theorem~\ref{thm:blowup-sing} we have that $\xi_0\in \Sigma_\kappa (p^{x_0}_\kappa)$. In particular,
$$
\partial^\alpha p^{x_0}_\kappa(\xi_0)=0,\quad\text{for }|\alpha|<\kappa.
$$
However, dividing both parts of \eqref{eq:compat-contr-nonzero} by $\rho_i^{\kappa-|\alpha|}$ and passing to the limit, we obtain that
$$
|\partial^\alpha p^{x_0}_\kappa(\xi_0)|=\Big|\sum_{\gamma>\alpha, |\gamma|=\kappa} \frac{a_\gamma(x_0)}{(\gamma-\alpha)!}(\xi_0)^{\gamma-\alpha}\Big|\geq \delta,
$$
a contradiction.
\end{proof}

We then apply Whitney's extension theorem and the implicit function theorem as in the proof of Theorem~\ref{thm:sing-points} to complete the proof.
\end{proof}

\section*{Concluding remarks and open problems}

The intent of this closing section is to provide a summarizing
overview of the state of our knowledge on the lower dimensional
obstacle problem. We also point to some open problems in the study
of the free boundary $\Gamma(u)$ for solutions of
\eqref{eq:signorini-1}--\eqref{eq:0-fbp}, and more generally for
\eqref{eq:signorini-1-nonzero}--\eqref{eq:0-fbp-nonzero}.

The regularity of $\Gamma_{2-\frac12}(u)$, i.e. of that portion of
the free boundary which is composed of regular points  has been
proved in \cite{ACS} for the zero obstacle and in \cite{CSS} for a
nonzero one. The present paper studies the singular set $\Sigma(u)$,
which is the collection of those free boundary points where the
coincidence set $\Lambda(u)=\{u=\phi\}$ has a vanishing
$\H^{n-1}$-density. 

What remains to study is the set of nonregular nonsingular points,
i.e. the set
$$
\Gamma(u)\setminus (\Gamma_{2-\frac12}(u)\cup
\Sigma(u))=\bigcup_{\substack{\kappa>2-\frac12,\\ \kappa\not=2m, m\in\N}}
\Gamma_\kappa(u).
$$
In the case of the nonzero thin obstacle $\phi\in C^{k,1}$ it is
reasonable to limit ourselves to $\kappa<k$, since at the free
boundary points in $\Gamma^{(k)}_k(u)$ even the blowups are not
properly defined.

\subsection*{Possible values of $\kappa$} First, one must identify the
possible values of $\kappa$ which can be reduced to a classification
of homogeneous global solutions of the Signorini problem with zero obstacle by Propositions~\ref{prop:blowup-homogen} and
\ref{prop:blowup-homogen-nonzero}. A partial classification of
global solutions (convex in $x'$) has been given in \cite{CSS} and
\cite{ACS}, which excluded the interval $(2-\frac12,2)$ from the
range of possible values of $\kappa$.
However, a full classification
is needed for obtaining the full range of possible values of
$\kappa$. As we mentioned earlier in the text it is plausible that
the only possible values are
$$
\kappa\in\{2m-\frac12 \mid m\in \mathbb N\}\cup\{2m \mid m\in \mathbb N\}.
$$
This fact is easy to establish when the dimension
$n=2$, see Remark~\ref{rem:poss_freq_dim2}, but is unknown in higher dimensions.

\subsection*{Free boundary}

We know that the only possible global solutions for $\kappa\in\{2m\mid  m\in\N\}$ are the polynomials $p_\kappa\in\P_\kappa$, see Lemma~\ref{lem:2m-homogen}. The other known global solutions are the rotations in $x'$-variable of
$$
\hat u_\kappa(x)=\Re(x_1+i\, |x_n|)^{\kappa},\quad \text{for}\quad\kappa\in\{2m-\frac12 \mid
m\in\N\}.
$$
If these are the only global solutions for this range of $\kappa$'s then it is plausible that
$$
\text{$\Gamma_\kappa(u)$ is locally a $(n-2)$-dimensional $C^1$-manifold
  for }\kappa\in\{2m-\frac12\mid m\in\N\},
$$
or at least Lipschitz regular. However, the true picture may be more complicated than that.

\subsection*{Degenerate points} One complication that may occur in dimensions
$n\geq 3$ is that a blowup $u_0$ may vanish identically on
$\R^{n-1}\times\{0\}$ for some $u\in\S$. This cannot happen in
dimension $n=2$, see Remark~\ref{rem:poss_freq_dim2}, however, the
possibility in higher dimensions is unknown to the authors. If a
blowup $u_0$ for $u\in\S$ vanishes on $\R^{n-1}\times\{0\}$, we
call the origin a \emph{degenerate free boundary point} of $u$. Such points
are characterized by the property
that the coincidence set $\Lambda(u)$ has a $\H^{n-1}$-density $1$ there.
At degenerate points, it is easy to see that one must have
$\kappa\in\{2m+1\mid m\in\N\}$ and that $u_0$ must coincide  in
$\R^{n-1}\times[0,\infty)$ with a harmonic polynomial $q_\kappa$ from the class
$$
\Q_\kappa=\{q_\kappa \mid \Delta q_\kappa=0,\ x\cdot\nabla
q_\kappa-\kappa q_\kappa=0,\
q_\kappa(x',0)=0,\  -\partial_{x_n} q_\kappa(x',0)\geq 0\}.
$$
It would be interesting to study the set of such points, provided they exist.

\end{document}